\documentclass[a4paper,11pt,twoside]{amsart}
\usepackage[english]{babel}
\usepackage[utf8]{inputenc}

\usepackage[a4paper,inner=3cm,outer=3cm,top=4cm,bottom=4cm,pdftex]{geometry}
\usepackage{fancyhdr}
\usepackage{color}
\usepackage{bold-extra}
\usepackage{ mathrsfs }

\usepackage{comment}
\usepackage{graphics}
\usepackage{aliascnt}
\usepackage[pdftex,citecolor=green,linkcolor=red]{hyperref}
\usepackage{mathtools}

\usepackage{amsmath}
\usepackage{amsfonts}
\usepackage{amssymb}
\usepackage{amsthm}
\usepackage{enumerate}

\newtheorem{theorem}{Theorem}[section]

\newtheorem{lemma}[theorem]{Lemma}
\newtheorem{proposition}[theorem]{Proposition}
\theoremstyle{definition}
\newtheorem{definition}[theorem]{Definition}
\newtheorem{remark}[theorem]{Remark}

\newtheorem{conjecture}[theorem]{Conjecture}
\numberwithin{equation}{section}

\newcommand\N{\mathbb{N}}

\newcommand\n{\textup{n}}

\renewcommand{\pmod}[1]{\ (\mathrm{mod}\ #1)}
\renewcommand\d{\textnormal{ d}}
\renewcommand\:{\colon}

\begin{document}

\author[J. Ter\"av\"ainen]{Joni Ter\"av\"ainen}
\address{Department of Mathematics and Statistics\\ University of Turku, 20014 Turku\\
Finland}
\email{joni.p.teravainen@gmail.com}

\title{Pointwise convergence of ergodic averages with M\"obius weight}

\begin{abstract} Let $(X,\nu,T)$ be a measure-preserving system, and let  $P_1,\ldots, P_k$ be polynomials with integer coefficients. We prove that, for any $f_1,\ldots, f_k\in L^{\infty}(X)$, the M\"obius-weighted polynomial multiple ergodic averages
\begin{align*}
\frac{1}{N}\sum_{n\leq N}\mu(n)f_1(T^{P_1(n)}x)\cdots f_k(T^{P_k(n)}x)    \end{align*}
converge to $0$ pointwise almost everywhere. Specialising to $P_1(y)=y, P_2(y)=2y$, this solves a problem of Frantzikinakis. We also prove pointwise convergence for a more general class of multiplicative weights for multiple ergodic averages involving distinct degree polynomials. For the proofs we establish some quantitative generalised von Neumann theorems for polynomial configurations that are of independent interest.
\end{abstract}

\maketitle

\section{Introduction}

Let $(X,\nu)$ be a probability space and $T\: X\to X$ an invertible measure-preserving map, meaning that $\nu(T^{-1}(E))=\nu(E)$ for all measurable  sets $E\subset X$. The triple $(X,\nu,T)$ is called a measure-preserving system. Given functions $f_1,\ldots, f_k\in L^{\infty}(X)$ and polynomials $P_1,\ldots, P_k\in \mathbb{Z}[y]$, we form the polynomial multiple ergodic averages
\begin{align}\label{eq:unweighted}
\frac{1}{N}\sum_{n\leq N}f_1(T^{P_1(n)})\cdots f_k(T^{P_k(n)}x),\quad x\in X.   \end{align}
The convergence properties of these averages as $N\to \infty$ have been studied intensively. The question of their $L^2(X)$ norm convergence was settled by Host--Kra~\cite{host-kra-polynomial} and  Leibman~\cite{leibman-polynomial} after a series of substantial progress by several authors. The question of pointwise convergence, however, remains open and is the subject of the celebrated Furstenberg--Bergelson--Leibman conjecture~\cite[Section 5.5]{bergelson-leibman-roth}. Pointwise convergence has so far been established in two notable cases, namely the case where $k=1$ and the case where $k=2$ and $P_1$ is linear. The former follows from celebrated work of Bourgain~\cite{Bourgain-polynomial}, and the latter follows from another well known work of Bourgain~\cite{Bourgain-twopoint} if $P_2$ is linear and from a recent breakthrough of Krause, Mirek and Tao~\cite{krause-mirek-tao} if $P_2$ has degree at least $2$.

Let $\mu$ be the M\"obius function, defined by $\mu(n)=(-1)^k$ if $n$ is the product of $k$ distinct primes and by $\mu(n)=0$ if $n$ is divisible by the square of a prime. In this paper, we shall consider the M\"obius-weighted polynomial multiple ergodic averages
\begin{align*}
\frac{1}{N}\sum_{n\leq N}\mu(n)f_1(T^{P_1(n)})\cdots f_k(T^{P_k(n)}x),\quad x\in X.   
\end{align*}
Based on the M\"obius randomness principle (see~\cite[Section 13]{iw-kow}), it is natural to conjecture the following.

\begin{conjecture}\label{conj1}
 Let $k\in \mathbb{N}$, and let  $P_1,\ldots, P_k$ be polynomials with integer coefficients. Let $(X,\nu,T)$ be a measure-preserving system.
Then, for any $f_1,\ldots, f_k\in L^{\infty}(X)$, we have 
\begin{align*}
\lim_{N\to \infty}\frac{1}{N}\sum_{n\leq N}\mu(n)f_1(T^{P_1(n)}x)\cdots f_k(T^{P_k(n)}x)=0
\end{align*}
for almost all $x\in X$. 
\end{conjecture}

The corresponding $L^2(X)$ norm convergence result was proven by Frantzikinakis and Host~\cite{FH-arithmeticsets} (this could also be deduced from the proof of~\cite[Theorem 1.3]{Chu}, combined with~\cite[Theorem 1.1]{gt-mobius}). The only case we are aware of where Conjecture~\ref{conj1} was previously known is the case $k=1$, proven in~\cite[Theorem 2.2]{eisner} (see also~\cite[Proposition 3.1]{EKLR} for the case of a linear polynomial and~\cite[Theorem C]{FH-arithmeticsets} for a generalisation of that result to other multiplicative functions). 

Our first main theorem settles Conjecture~\ref{conj1} in full. In fact, somewhat unusually for ergodic theorems, we get a quantitative (polylogarithmic) rate of converge. This result goes beyond what is currently known about the unweighted averages~\eqref{eq:unweighted}.

\begin{theorem}\label{thm_polyergodic}
 Let $k\in \mathbb{N}$, and let  $P_1,\ldots, P_k$ be polynomials with integer coefficients. Let $1<q_1,\ldots, q_k\leq \infty$ satisfy $\frac{1}{q_1}+\cdots+\frac{1}{q_k}< 1$. Let $(X,\nu,T)$ be a measure-preserving system. Then, for any $f_1\in L^{q_1}(X),\ldots, f_k\in L^{q_k}(X)$ and $A\geq 0$, we have
\begin{align*}
\lim_{N\to \infty}\frac{(\log N)^{A}}{N}\sum_{n\leq N}\mu(n)f_1(T^{P_1(n)}x)\cdots f_k(T^{P_k(n)}x)=0   \end{align*}
for almost all $x\in X$. 
\end{theorem}

Let us make a few remarks about~Theorem~\ref{thm_polyergodic}. 

\begin{enumerate}

    \item Specialising to $k=2$ and $P_1(y)=y, P_2(y)=2y$ (and $A=0$, $q_1=q_2=\infty$), Theorem~\ref{thm_polyergodic} settles the M\"obius case of Problem 12 of Frantzikinakis' open problems survey~\cite{Frantzikinakis-open}\footnote{Frantzikinakis also asked about the convergence of the same multiple ergodic averages weighted by any multiplicative function $g\colon \mathbb{N}\to \mathbb{C}$, taking values in the unit disc, which has a mean value in every arithmetic progression. Theorem~\ref{thm_multpolyergodic} below makes progress on this more general question.}. This problem was also stated by Frantzikinakis and Host in~\cite{FH-arithmeticsets}.

   \item In the case where the $P_i$ are linear, we can allow iterates of different commuting transformations in the result; see Theorem~\ref{thm_commuting} below.

    \item It is likely that, at least for $k=2$ and $P_1$ linear, the region of $(q_1,\ldots, q_k)$ in the theorem could be improved to $\frac{1}{q_1}+\cdots +\frac{1}{q_k}\leq 1+\delta_{d}$ for some $\delta_{d}>0$, hence ``breaking duality'' in this problem. This is thanks to the $L^p$-improving estimates of Lacey~\cite{lacey} (see also~\cite{oberlin}) and Han--Kova\v{c}--Lacey--Madrid--Yang~\cite{HKLMY}. See also~\cite[Section 11]{krause-mirek-tao}. We leave the details to the interested reader.

        \item The theorem continues to hold if we replace the M\"obius function $\mu$ with the Liouville function $\lambda$ (defined by $\lambda(n)=(-1)^{\Omega(n)}$, where $\Omega(n)$ is the number of prime factors of $n$ with multiplicities); see the more general Theorem~\ref{thm_polyergodic-general} along with Remark~\ref{rmk-liouville}. Similarly, all the results in this paper regarding the M\"obius function also hold for the Liouville function. 

\end{enumerate}

\subsection{Ergodic averages with commuting transformations}

Let a probability space $(X,\nu)$ and invertible, commuting, measure-preserving maps $T_1,\ldots, T_k\colon X\to X$ be given. For any polynomials $P_1,\ldots, P_k$ with integer coefficients and $f_1,\ldots, f_k\in L^{\infty}(X)$, one can  consider polynomial ergodic averages with commuting transformations
\begin{align}\label{eq:commute}
 \frac{1}{N}\sum_{n\leq N}f_1(T_1^{P_1(n)}x)\cdots f_k(T_k^{P_k(n)}x),\quad x\in X   
\end{align}
and their M\"obius-weighted versions
\begin{align}\label{eq:mucommute}
  \frac{1}{N}\sum_{n\leq N}\mu(n)f_1(T_1^{P_1(n)}x)\cdots f_k(T_k^{P_k(n)}x),\quad x\in X.     
\end{align}
The commuting case seems to be rather more difficult, since in the unweighted case~\eqref{eq:commute} pointwise convergence is not currently known even for $k=2$ and $P_1,P_2$ linear (see Problem 19 of~\cite{Frantzikinakis-open}). However, we mention that Walsh~\cite{walsh-annals} proved $L^2(X)$ norm convergence of the averages~\eqref{eq:commute} in a groundbreaking work, and the Furstenberg--Bergelson--Leibman conjecture asserts that pointwise convergence should hold also with commuting transformations. 

Our next theorem states that for the M\"obius averages~\eqref{eq:mucommute} we have pointwise convergence in the case where the $P_i$ are linear. 

\begin{theorem}\label{thm_commuting}
   Let $k\in \mathbb{N}$. Let $(X,\nu)$ be a probability space and let $T_1,\ldots, T_k\colon X\to X$ be invertible, commuting, measure-preserving maps. Let $1<q_1,\ldots, q_k\leq \infty$ satisfy $\frac{1}{q_1}+\cdots+\frac{1}{q_k}<1$, and let $f_1\in L^{q_1}(X),\ldots, f_k\in L^{q_k}(X)$. Then, for any $A\geq 0$, we have
   \begin{align*}
   \lim_{N\to \infty}\frac{(\log N)^{A}}{N}\sum_{n\leq N}\mu(n)f_1(T_1^n x)\cdots f_k(T_k^n x)=0    
   \end{align*}
   for almost all $x\in X$.
\end{theorem}

Naturally, this theorem implies Theorem~\ref{thm_polyergodic} for $P_i$ being linear by taking to $T_1,\ldots, T_k$ being powers of the same transformation.

 It seems likely that also the general case of Theorem~\ref{thm_polyergodic} could be obtained for commuting transformations by extending Theorem~\ref{thm:GVNT1}, which goes into its proof, to  multivariate functions (which could likely be done with a more complicated PET induction scheme). We leave the details to the interested reader.

\subsection{Multiplicative weights}

We also consider more general weighted polynomial multiple ergodic averages 
\begin{align}\label{eq:gaverage}
\frac{1}{N}\sum_{n\leq N}g(n)f_1(T^{P_1(n)})\cdots f_k(T^{P_k(n)}x),\quad x\in X,   
\end{align}
with $g\colon \mathbb{N}\to \mathbb{C}$ a function. Already for $k=1$ and $P_1$ linear, a necessary condition for the pointwise convergence of these averages is that $g$ has \emph{convergent means}, meaning that $\lim_{N\to \infty}\frac{1}{N}\sum_{n\leq N}g(an+b)$ exists for all $a,b\in \mathbb{N}$ (this is seen by taking $X$ to be a finite set).

In Theorem~\ref{thm_polyergodic-general} we will show, assuming a mild growth condition on $g$, that good decay bounds on the Gowers $U^s[N]$ norms of $g$ (or for certain weaker $u^s[N]$ norms that depend on the maximal correlation with polynomial phases) imply the convergence of the averages~\eqref{eq:gaverage} to $0$. Such results should have applications also in cases where $g$ has no arithmetic structure (for example, for random weights); however, here we focus on applications with $g$ being multiplicative\footnote{We say that $g\colon \mathbb{N}\to \mathbb{C}$ is multiplicative if $g(mn)=g(m)g(n)$ whenever $m,n$ are coprime}.

Frantzikinakis~\cite{Frantzikinakis-open}  conjectured that if $g\colon \mathbb{N}\to \mathbb{C}$ is $1$-bounded, multiplicative and hasconvergent means, then we have the pointwise almost everywhere convergence of
\begin{align*}
\frac{1}{N}\sum_{n\leq N}g(n)f_1(T^nx)f_2(T^{2n}x).     
\end{align*}
There is an obvious extension of this conjecture to polynomial multiple ergodic averages.

\begin{conjecture}\label{conj1-mult}
 Let $k\in \mathbb{N}$, and let  $P_1,\ldots, P_k$ be polynomials with integer coefficients. Let $g\colon \mathbb{N}\to \mathbb{C}$ be a multiplicative function taking values in the unit disc and having convergent means. Let $(X,\nu,T)$ be a measure-preserving system.
Then, for any $f_1,\ldots, f_k\in L^{\infty}(X)$, the limit
\begin{align*}
\lim_{N\to \infty}\frac{1}{N}\sum_{n\leq N}g(n)f_1(T^{P_1(n)}x)\cdots f_k(T^{P_k(n)}x)
\end{align*}
exists for almost all $x\in X$. 
\end{conjecture}

While we are not able to prove this statement in full, we can prove it for a natural class of multiplicative functions, namely those satisfying a \emph{Siegel--Walfisz assumption} (stated below). Most practically occurring multiplicative functions of mean $0$ satisfy this property, and this class arises naturally in several problems in analytic number theory, in particular in connection with the Bombieri--Vinogradov theorem (see e.g.~\cite{GS-1}). 

\begin{definition}\label{def:SW}
We say that a function $g\colon \mathbb{N}\to \mathbb{C}$ satisfies the \emph{Siegel--Walfisz} assumption if the following hold:
\begin{enumerate}
    \item $g$ is \emph{divisor-bounded}: for some $C\geq 0$, we have $|g(n)|\leq d(n)^{C}$ for all $n\in \mathbb{N}$, with $d(n)$ denoting the number of positive divisors of $n$.    

    \item  For all $A>0$ and $N\geq 3$ we have 
\begin{align*}
    \max_{1\leq a\leq q\leq (\log N)^{A}}\left|\sum_{\substack{n\leq N\\n\equiv a\pmod q}}g(n)\right|\ll_A \frac{N}{(\log N)^{A}}.
\end{align*}    
\end{enumerate}
\end{definition}

Examples of multiplicative functions satisfying the Siegel--Walfisz assumption include $h(n)d(n)^j\chi(n)n^{it}$ for any integer $j\geq 0$, real $t$ and Dirichlet character $\chi$, where $h\colon \mathbb{N}\to \mathbb{C}$ is any  bounded multiplicative function  that is ``pretending'' to be the M\"obius function in the sense that $\sum_p \frac{1+\textnormal{Re}(h(p))}{p}<\infty$.\footnote{That these functions are examples can be verified by using Perron's formula and standard estimates for Dirichlet $L$-functions close to the $1$-line.}

We are now ready to state a result on the pointwise convergence of multiple ergodic averages with a multiplicative weight satisfying the Siegel--Walfisz assumption. 

\begin{theorem}\label{thm_multpolyergodic}
 Let $k\in \mathbb{N}$, and let  $P_1,\ldots, P_k$ be polynomials with integer coefficients and with distinct degrees. Let $1<q_1,\ldots, q_k\leq \infty$ satisfy $\frac{1}{q_1}+\cdots+\frac{1}{q_k}< 1$. Let $(X,\nu,T)$ be a measure-preserving system. Let $g\colon \mathbb{N}\to \mathbb{C}$ be a multiplicative function satisfying the Siegel--Walfisz assumption. Then, for any $f_1\in L^{q_1}(X),\ldots, f_k\in L^{q_k}(X)$, we have
\begin{align*}
\lim_{N\to \infty}\frac{1}{N}\sum_{n\leq N}g(n)f_1(T^{P_1(n)}x)\cdots f_k(T^{P_k(n)}x)=0   \end{align*}
for almost all $x\in X$. 
\end{theorem}

Note that we allow the function $g$ to be unbounded, hence proving pointwise convergence in some cases not covered by Conjecture~\ref{conj1-mult}.

Somewhat curiously, Theorem~\ref{thm_multpolyergodic} requires an argument that is rather different from  our proof of Theorem~\ref{thm_polyergodic}; both proofs will be discussed in Section~\ref{sec:ideas}. 

\subsection{Prime ergodic averages}

The arguments presented in this paper are not limited to polynomial ergodic averages weighted by the M\"obius function, and indeed apply to similar ergodic averages with any weight that satisfies certain Gowers uniformity assumptions as well as some weak upper bound assumptions that are easy to verify; see Theorem~\ref{thm_polyergodic-general}. In particular, thanks to the quantitative Gowers uniformity estimates for the von Mangoldt function in~\cite{lengII}, these general theorems can be applied to reduce the problem of convergence of polynomial ergodic averages weighted by the primes to the problem of convergence of the same averages weighted by integers with no small prime factors, which is an easier problem (though still highly nontrivial). Pointwise convergence of polynomial multiple ergodic averages weighted by the primes will be studied in a future joint work with Krause, Mousavi and Tao.

\subsection{Further applications of the proof method}

Key ingredients in the proofs of our main theorems are some new polynomial generalised von Neumann theorems with quantitative dependencies that we establish in Section~\ref{sec:GVNT}. These results are likely to have applications also to other problems, such as to bounds for sets of integers lacking progressions of the form $x,x+P_1(p-1),\ldots x+P_k(p-1)$ with $p$ prime and with $P_1,\ldots, P_k$ polynomials of distinct degrees with $P_i(0)=0$. Such applications will be investigated in future works.

\subsection{Acknowledgements} The author thanks Nikos Frantzikinakis, Ben Krause, Sarah Peluse, Sean Prendiville and Terence Tao for helpful discussions and suggestions. 
The author was supported by funding from European Union's Horizon
Europe research and innovation programme under Marie Sk\l{}odowska-Curie grant agreement No
101058904.

\section{Proof ideas}\label{sec:ideas}

We now give an overview of the arguments used to prove Theorems~\ref{thm_polyergodic} and~\ref{thm_multpolyergodic}, presenting the steps of the proof in a somewhat different order than in the actual proof and focusing on the case of functions $f_i\in L^{\infty}(X)$ for simplicity.

We begin with Theorem~\ref{thm_polyergodic}. Let $P_1,\ldots, P_k$ be polynomials with integer coefficients and with highest degree $d$.
Thefirst step for the proofs of our pointwise convergence results is  a lacunary subsequence trick, which combined with the Borel--Cantelli lemma and Markov's inequality reduces matters to obtaining strong quantitative pointwise bounds of the form
\begin{align}\label{eq:musum}\begin{aligned}
&\left|\frac{1}{N^{d+1}}\sum_{m\leq N^d}\sum_{n\leq N}\mu(n)\phi(T^m x)f_1(T^{m+P_1(n)}x)\cdots f_k(T^{m+P_k(n)}x)\right|\\
&\quad \ll \|\phi\|_{L^{\infty}(X)} \|f_1\|_{L^{\infty}(X)}\cdots \|f_k\|_{L^{\infty}(X)}(\log N)^{-A} 
\end{aligned}
\end{align}
for $\nu$-almost all $x$ and all $\phi\in L^{\infty}(X)$, with $A$ large enough. This reduction is presented in Section~\ref{sec:mainproof}. 

We then establish  a polynomial generalised von Neumann theorem for counting operators on the left-hand side of~\eqref{eq:musum}, which bounds the averages in~\eqref{eq:musum} in terms of the $U^s[N]$ Gowers norm of $\mu$ for some $s\in \mathbb{N}$; see Theorem~\ref{thm:GVNT1}. This result is proven by repeated applications of van der Corput's inequality coupled with the PET induction scheme. Crucially, the bounds we obtain for~\eqref{eq:musum} in terms of the Gowers norm $\|\mu\|_{U^s[N]}$ are quantitative with polynomial (in fact, linear) dependencies.

After establishing this generalised von Neumann theorem, we can conclude the proof by applying the strongest known quantitative bounds for the $U^s[N]$ norms of the M\"obius function (see Lemma~\ref{le:leng}), which save an arbitrary power of logarithm, thanks to recent work of Leng~\cite{lengII} that builds on the work of Leng--Sah--Sawhney~\cite{LSS}.

In the case of Theorem~\ref{thm_multpolyergodic}, we repeat the lacunary subsequence argument to reduce to~\eqref{eq:musum}, with $\mu$ replaced by a multiplicative function $g$ satisfying the Siegel--Walfisz assumption. If one were now to apply Theorem~\ref{thm:GVNT1} again, one would not be able to obtain a sufficiently strong bound on the $U^s[N]$ norm of $g$, since the Leng--Sah--Sawhney inverse theorem~\cite{LSS} is quasipolynomial rather than polynomial. We overcome this by establishing a different polynomial generalised von Neumann theorem, Theorem~\ref{thm:GVNT2}, that (perhaps unexpectedly at first) allows bounding~\eqref{eq:musum} in terms of a weaker norm than the $U^s[N]$ norm of the weight function. This weaker norm, called the $u^s[N]$ norm and defined in~\eqref{eq:us}, expresses the maximal correlation of the weight $g$ with a polynomial phase of degree at most $s-1$. The proof of Theorem~\ref{thm:GVNT2} draws motivation from the Peluse--Prendiville degree lowering theory~\cite{Peluse-FMP},~\cite{PP-Roth}. The proof proceeds by induction on the length of the progression and involves showing that the first two functions in a weighted progressions can be assumed to be ``locally linear'' phase functions in a suitable sense. This conclusion is then boosted to global linearity with some extra work, which allows reducing the length of the progression, hence completing the induction.

Since the $u^s[N]$ norm already involves correlations with polynomial phases, we are able to bypass the need for the inverse theorem for the $U^s[N]$ norm when working with distict degree polynomials. In Subsection~\ref{sec:Siegel}, we show (using in particular a restriction to typical factorisations and bilinear estimates for polynomial exponential sums) that if $g$ is multiplicative and satisfies the Siegel--Walfisz assumption, then $g$ is close in $L^1[N]$ norm to a function $\widetilde{g}$ whose $u^s[N]$ norms decay faster than any power of logarithm. This together with Theorem~\ref{thm:GVNT2} mentioned above suffices for concluding the proof of Theorem~\ref{thm_multpolyergodic}. The Siegel--Walfisz assumption gives just the right decay for~\eqref{eq:musum}: with any weaker assumption we would not be able to prove this (although the averages~\eqref{eq:gaverage} should still converge).

We lastly remark that the approach based on Theorem~\ref{thm:GVNT2} also gives a different and arguably simpler proof of Theorem~\ref{thm_polyergodic} for distinct degree polynomials that is independent of any inverse theorems for the Gowers norms  and only uses classical analytic number theory input (the only property needed of the M\"obius function is an exponential sum estimate that essentially goes back to the work of Vinogradov~\cite{vino} from 1939; see Remark~\ref{rmk:vin}).

\section{Notation and preliminaries}\label{sec:notation} 

\subsection{Asymptotic notation, indicators and averages}

We use the Vinogradov and Landau asymptotic notations $\ll, \gg, o(\cdot), O(\cdot)$. Thus, we write $X\ll Y$, $X=O(Y)$ or $Y\gg X$ if there is a constant $C$ such that $|X|\leq CY$. We use $X\asymp Y$ to denote $X\ll Y\ll X$. We write $X=o(Y)$ as $N\to \infty$ if $|X|\leq c(N)Y$ for some function $c(N)\to 0$ as $N\to \infty$. If we add subscripts to these notations, then the implied constants can depend on these subscripts. Thus, for example  $X\ll_A Y$ means that $|X|\leq C_AY$ for some $C_A>0$ depending on $A$. 

For a set $E$, we define the indicator function $1_{E}(x)$ as the function that equals to $1$ if $x\in E$ and equals to $0$ otherwise. Similarly, if $P$ is a proposition, the expression $1_{P}$ equals to $1$ if $P$ is true and $0$ if $P$ is false. 

For a nonempty finite set $A$ and a function $f\:A\to \mathbb{C}$, we define the averages 
\begin{align*}
\mathbb{E}_{a\in A}f(a)\coloneqq \frac{\sum_{a\in A}f(a)}{\sum_{a\in A}1}.    
\end{align*}

For a real number $N\geq 1$, we denote $[N]\coloneqq \{n\in \mathbb{N}\:\,\, n\leq N\}$. For integers $m,n$, we denote their greatest common divisor by $(m,n)$ and write $m\mid n^{\infty}$ to mean that $m\mid n^k$ for some natural number $k$. Unless otherwise specified, all our sums and averages run over the positive integers, with the exception that the symbol $p$ is reserved for primes. 

\subsection{Gowers norms}

For $s\in \mathbb{N}$ and a function $f\: \mathbb{Z}\to \mathbb{C}$ with finite support, we define its unnormalised $U^s$ Gowers norm as
\begin{align*}
\|f\|_{\widetilde{U}^s(\mathbb{Z})}\coloneqq\left(\sum_{x,h_1,\ldots, h_s\in \mathbb{Z}}\prod_{\omega\in \{0,1\}^s}\mathcal{C}^{|\omega|}f(x+\omega\cdot (h_1,\ldots, h_s))\right)^{1/2^s},   
\end{align*}
where $\mathcal{C}(z)=\overline{z}$ is the complex conjugation operator and for a vector $(\omega_1,\ldots, \omega_s)$ we write $|\omega|\coloneqq |\omega_1|+\cdots+|\omega_s|$. 
For $N\geq 1$, we then define the $U^s[N]$ Gowers norm of a function $f\:\mathbb{Z}\to \mathbb{C}$ as
\begin{align*}
\|f\|_{U^s[N]}\coloneqq \frac{\|f1_{[N]}\|_{\widetilde{U}^s(\mathbb{Z})}}{\|1_{[N]}\|_{\widetilde{U}^s(\mathbb{Z})}}.    
\end{align*}
As is well known (see for example~\cite[Appendix B]{gt-linear}), for $s\geq 2$ the $U^s[N]$ norm is indeed a norm and for $s=1$ it is a seminorm, and the function $s\mapsto \|f\|_{U^s[N]}$ is increasing. 

We observe the classical $U^2[N]$ inverse theorem: if $f\:[N]\to \mathbb{C}$ satisfies $|f|\leq 1$ and $\delta\in (0,1)$, then 
\begin{align}\label{eq:U2}
\|f\|_{U^2[N]}\geq \delta\implies \sup_{\alpha \in \mathbb{R}}|\mathbb{E}_{n\in [N]}f(n)e(-\alpha n)|\gg \delta^2,     
\end{align}
where $e(x)\coloneqq e^{2\pi i x}$. 
This follows from the identity $\|f\|_{\widetilde{U}^2(\mathbb{Z})}=\int_{0}^{1}|\sum_{n\in \mathbb{Z}}f(n)e(-\alpha n)|^4\d \alpha$ (which can be verified by expanding out the right-hand side) combined with Parseval's identity.

For the $\widetilde{U}^s(\mathbb{Z})$ norms we have the Gowers--Cauchy--Schwarz inequality (see for example~\cite[(4.2)]{TT-JEMS}), which states that, for any functions $(f_{\omega})_{\omega\in \{0,1\}^s}$ from $\mathbb{Z}$ to $\mathbb{C}$ with finite support, we have
\begin{align}\label{eq:gcs}
\left|\sum_{x,h_1,\ldots, h_s\in \mathbb{Z}}\prod_{\omega\in \{0,1\}^s}f_{\omega}(x+\omega\cdot (h_1,\ldots, h_s))\right|\leq \prod_{\omega\in \{0,1\}^s}\|f_{\omega}\|_{\widetilde{U}^s(\mathbb{Z})}.  
\end{align}

We also define the $u^s[N]$ norm of a function $f\:\mathbb{Z}\to \mathbb{C}$ as
\begin{align}\label{eq:us}
\|f\|_{u^s[N]}\coloneqq \sup_{\substack{P(y)\in \mathbb{R}[y]\\\deg(P)\leq s-1}}\left|\mathbb{E}_{n\in [N]}f(n)e(-P(n))\right|.   \end{align}
We remark that, as is well known, for $s=2$ the $u^s[N]$ and $U^s[N]$ norms are equivalent (up to polynomial losses), but for all $s\geq 3$ the $u^s[N]$ norm is weaker in the sense that the $U^s[N]$ norm controls the $u^s[N]$ norm but not vice versa; see~\cite{gtz},~\cite[Section 4]{gowers-GAFA}.

\subsection{Van der Corput's inequality}

For $H\geq 1$, define the weight
\begin{align*}
\mu_H(h)\coloneqq \frac{|\{(h_1,h_2)\in [H]^2\: h_1-h_2=h\}|}{\lfloor H\rfloor ^2},    
\end{align*}
where $\lfloor x\rfloor$ is the floor function of $x$ (This weight should not be confused with the  M\"obius function $\mu$). For an integer $h$, we define the differencing operator $\Delta_h$ by setting $\Delta_hf(x)=\overline{f(x+h)}f(x)$ for any $f\: \mathbb{Z}\to \mathbb{C}$ and $x\in \mathbb{Z}$.

We will frequently use van der Corput's inequality in the following form.

\begin{lemma}\label{le:vdc}
For any $N\geq H\geq 1$ and any function $f\:\mathbb{Z}\to \mathbb{C}$ supported on $[N]$, we have
\begin{align}\label{eq:vdc}
\left|\mathbb{E}_{n\in [N]}f(n)\right|^2\leq \frac{N+H}{\lfloor N\rfloor} \sum_{h\in \mathbb{Z}}\mu_H(h)\mathbb{E}_{n\in [N]}\Delta_h f(n).   
\end{align}    
\end{lemma}
\begin{proof}
 See for example~\cite[Lemma 3.1]{Prendiville}.
\end{proof}

\subsection{Vinogradov's Fourier expansion}

For a real number $x$, we write $\|x\|$ for the distance from $x$ to the nearest integer(s).

We shall need a Fourier approximation for the indicator function of an interval that goes back to Vinogradov.

\begin{lemma}\label{le:vinogradov} For any real numbers $-1/2\leq \alpha<\beta\leq 1/2$ and $\eta\in (0,\min\{1/2-\|\alpha\|,1/2-\|\beta\|,\|\alpha-\beta\|/2\})$, there exists a $1$-periodic function $g\:\mathbb{R}\to [0,1]$ with the following properties.
\begin{enumerate}
    \item $g(x)=1$ for $x\in [\alpha+\eta, \beta-\eta]$, $g(x)=0$ for $x\in [-1/2,1/2]\setminus [\alpha-\eta,\beta+\eta]$, and $0\leq g(x)\leq 1$ for all $x\in [-1/2,1/2]$.

    \item For some $|c_j|\leq 10\eta$, we have the pointwise convergent Fourier representation
    \begin{align*}
     g(x)=\beta-\alpha-\eta+\sum_{|j|>0}c_je(jx).   
    \end{align*}
    \item For any $K\geq 1$, we have
    \begin{align*}
     \sum_{|j|>K}|c_j|\leq \frac{10\eta^{-1}}{K}.   
    \end{align*}
\end{enumerate}
\end{lemma}

\begin{proof}
 This follows from~\cite[Lemma 12]{vinogradov} (taking $r=1$ there).   
\end{proof}

\section{Polynomial generalised von Neumann theorems}\label{sec:GVNT}

In this section, we prove generalised von Neumann theorems for the weighted polynomial counting operators
\begin{align}
 \label{eq:Lambda}\Lambda^{M,N}_{P_1,\ldots, P_k}(\theta;f_0,f_1,\ldots, f_k)\coloneqq \frac{1}{MN}\sum_{m\in \mathbb{Z}}\sum_{n\in [N]}\theta(n)f_0(m)f_1(m+P_1(n))\cdots f_k(m+P_k(n)), \end{align}
 where $P_1,\ldots, P_k\in \mathbb{Z}[y]$ and $\theta\:[N]\to \mathbb{C}$ is a weight function (which in applications we take  to be a multiplicative function) and $f_0,f_1,\ldots, f_k\: \mathbb{Z}\to \mathbb{C}$ are functions supported on $[-M,M]$ (with $M\asymp N^{\max_{j}(\deg P_j)}$). Thus, we bound the expression~\eqref{eq:Lambda} in terms of some Gowers norm (or related norm) of $\theta$ or $f_0$. It is important for the proofs of our main theorems that the obtained results are quantitative, with polynomial dependencies. The two main results of this section (of independent interest) are  Theorems~\ref{thm:GVNT2} and~\ref{thm:GVNT1}; they are used for proving Theorems~\ref{thm_multpolyergodic} and~\ref{thm_polyergodic}, respectively.

The first main result of this section states that if $P_1,\ldots, P_k$ have distinct degrees, then the polynomial counting operator~\eqref{eq:Lambda} is bounded in terms of the $u^s[N]$ norm of the weight $\theta$ for some $s$, with polynomial dependencies. It is important to have the $u^s[N]$ norm rather than the $U^s[N]$ norm here, since for the $U^s[N]$ norm we do not currently have a polynomial inverse theorem, meaning that the Siegel--Walfisz assumption from Definition~\ref{def:SW} would be insufficient if we only had a bound in terms of these norms.

\begin{theorem}[A polynomial generalised von Neumann theorem with $u^s$ control]\label{thm:GVNT2} Let $d,k\in \mathbb{N}$ and $C\geq 1$. Let  $P_1,\ldots, P_k$ be a polynomials with integer coefficients satisfying $\deg P_1<\deg P_2<\cdots<\deg P_k=d$.  Let $N\geq 1$, and let $f_0,\ldots,f_k\: \mathbb{Z}\to \mathbb{C}$ be functions supported on $[-CN^d,CN^d]$ with $|f_i|\leq 1$ for all $0\leq i\leq k$, and let $\theta\:[N]\to \mathbb{C}$ be a function with $|\theta|\leq 1$. Then, for some $1\leq K\ll_{d}1$, we have
\begin{align*}
\left|\frac{1}{N^{d+1}}\sum_{m\in \mathbb{Z}}\sum_{n\in [N]}\theta(n)f_0(m)f_1(m+P_1(n))\cdots f_k(m+P_k(n))\right|\ll_{C,P_1,\ldots, P_k} (N^{-1}+\|\theta\|_{u^{d+1}[N]})^{1/K}.    
\end{align*}
\end{theorem}

The proof of Theorem~\ref{thm:GVNT2} is given in the next three subsections. In  Subsection~\ref{sub:transferlinear}, we show that $f_0,f_1$ can be assumed to be ``locally linear phase functions'' in a suitable sense. In Subsection~\ref{sub:circle}, we prove the $k=1$ case of the theorem using the circle method; this works as a base case for the proof which is by induction on $k$. Finally, in Subsection~\ref{sub:conclude}, we use the conclusions of the preceding subsections together with an iterative argument for the function $f_1$ to conclude the proof.

The second main result of this section is that, for any polynomials $P_1,\ldots, P_k$, the polynomial counting operator~\eqref{eq:Lambda} is always bounded by some $U^s[N]$ norm of the weight $\theta$, with linear dependence on the Gowers norm. In what follows, for any finite nonempty collection $\mathcal{Q}$ of polynomials with integer coefficients, we define its \emph{degree} $\deg \mathcal{Q}$ as the largest of the degrees of the polynomials in $\mathcal{Q}$.

\begin{theorem}[A  polynomial generalised von Neumann theorem with $U^s$ control]\label{thm:GVNT1} Let $d\in \mathbb{N}$ and $C\geq 1$. Let $N\geq 1$, and let $\theta\:[N]\to \mathbb{C}$ be a function.  Let $\mathcal{Q}$ be a finite collection of polynomials with integer coefficients satisfying $\deg \mathcal{Q}=d$ and $Q([N])\subset [-CN^d,CN^d]$ for all $Q\in \mathcal{Q}$. For each $Q\in \mathcal{Q}$ let $f_Q\: \mathbb{Z}\to \mathbb{C}$ be a function supported on $[-CN^d,CN^d]$ with $|f_Q|\leq 1$. Then, for some natural number $s\ll_{|\mathcal{Q}|,\deg \mathcal{Q}}1$, we have
\begin{align}\label{eq:GVNT1-1poly}
\left|\frac{1}{N^{d+1}}\sum_{m\in \mathbb{Z}}\sum_{n\in [N]}\theta(n)\prod_{Q\in \mathcal{Q}}f_Q(m+Q(n))\right|\ll_{|\mathcal{Q}|,\deg \mathcal{Q},C} \|\theta\|_{U^{s}[N]}.    
\end{align} 
\end{theorem}

The proof of Theorem~\ref{thm:GVNT1} is based on the PET induction scheme and is given in Subsection~\ref{sub:genpoly}. We also present there a multidimensional version of the special case where $\deg \mathcal{Q}=1$ (Lemma~\ref{le:GVNT-multi}); this will be needed for the proof of Theorem~\ref{thm_commuting}.

\subsection{Transferring to locally linear functions}\label{sub:transferlinear}

The first step in the proof of Theorem~\ref{thm:GVNT2} is to show that if a polynomial average of the form~\eqref{eq:Lambda} is large, then the functions $f_0,f_1$ can be assumed to be locally linear phase functions. In what follows, we say that a function $\phi\colon \mathbb{Z}\to \mathbb{C}$ is a \emph{locally linear phase function of resolution $M$} is for some real numbers $\alpha_m$ we have $\phi(m)=e(\alpha_mm)$ for all $m\in \mathbb{Z}$ and if additionally
 there is a partition of $\mathbb{Z}$ into discrete intervals of length $M$ such that $m\mapsto \alpha_m$ is constant on the cells of that partition. We call the set $\{\alpha_m\pmod 1\colon m\in \mathbb{Z}\}$ the \emph{spectrum} of $\phi$.

\begin{proposition}[Reduction to locally linear phases]\label{prop_loclin}
 Let $C\geq 1$, $d,k\in \mathbb{N}$, and let $P_1,\ldots, P_k$ be polynomials with integer coefficients and with $\deg P_1<\cdots <\deg P_k=d$. Let $N\geq 1$, and let $f_0,f_1,\ldots, f_k\colon \mathbb{Z}\to \mathbb{C}$ be functions supported on $[-CN^d, CN^d]$ and with $|f_i|\leq 1$ for all $0\leq i\leq k$. Let $\theta\colon [N]\to \mathbb{C}$ be a function with $|\theta|\leq 1$. Then, for some $1\leq K\ll_d 1$, we have 
 \begin{align*}
 &\left|\frac{1}{N^{d+1}}\sum_{m\in \mathbb{Z}}\sum_{n\in [N]}\theta(n)f_0(m)f_1(m+P_1(n))\cdots f_k(m+P_k(n))\right|\\
 &\quad \ll_{C,P_1,\ldots, P_k} \left(N^{-1}+\left|\frac{1}{N^{d+1}}\sum_{m\in \mathbb{Z}}\sum_{n\in [N]}\theta(n)\phi_0(m)\phi_1(m+P_1(n))\cdots f_k(m+P_k(n))\right|\right)^{1/K}       
 \end{align*}
 for some locally linear phase functions $\phi_0,\phi_1$ of resolution $\gg_{C,P_1,\ldots, P_k}\delta^{O_d(1)}N^{\deg P_1}$. Moreover, we may assume that the spectra of $\phi_0,\phi_1$ belong to $\frac{1}{N^{\deg P_1}}\mathbb{Z}$. 
\end{proposition}

Proposition~\ref{prop_loclin} may be compared with, and is motivated by, the work of Peluse and Prendiville~\cite[Theorem 1.5]{PP-Roth}, where in the case $k=2$, $\theta\equiv 1$ and $P_1(y)=y, P_2(y)=y^2$, it is proven that $f_0,f_1,f_2$ can be replaced more strongly with major arc locally linear phase functions. In the more general setup of Proposition~\ref{prop_loclin}, it is not possible to reduce to major arc locally linear phase functions.

For the proof of Proposition~\ref{prop_loclin}, we need Peluse's inverse theorem.

 \begin{lemma}[Peluse's inverse theorem]\label{le:degree-lower} Let $k,d_1,d\in \mathbb{N}$ and $C\geq 1$. Let $P_1,\ldots, P_k$ be polynomials with integer coefficients satisfying $P_i(0)=0$ for all $1\leq i\leq k$ and $d_1=\deg P_1<\ldots<\deg P_k=d$, and with all the coefficients of the polynomials $P_i$ being bounded by $C$ in modulus. 

Let $N\geq 1$ and $\delta\in (0,1/2)$. Let $f_1,\ldots, f_k\: \mathbb{Z}\to \mathbb{C}$ be  functions supported on $[-CN^d,CN^d]$ with $|f_i|\leq 1$ for all $0\leq i\leq k$. Then there exists $1\leq B\ll_{d}1$ such that for either $j\in \{0,1\}$ we have
\begin{align}\label{eq:peluse-inverse}\begin{split}
&\left|\frac{1}{N^{d+1}}\sum_{m\in \mathbb{Z}}\sum_{n\in [N]}f_0(m)f_1(m+P_1(n))\cdots f_k(m+P_k(n))\right|\\
&\quad \ll_{C,d} \delta+\left(N^{-1}+\max_{\substack{q\in [\delta^{-B}]\\N'\in [\delta^{B}N^{d_1}, N^{d_1}]}}\frac{1}{N^{d}}\sum_{m\in \mathbb{Z}}\left|\frac{1}{N'}\sum_{n\in [N']}f_j(m+qn)\right|\right)^{1/B}. 
\end{split}
\end{align}
\end{lemma}

\begin{proof} This will follow from~\cite[Theorem 3.3]{Peluse-FMP} after some reductions.
 It suffices to show that for each $j\in \{0,1\}$ there exists $1\leq B\ll_{d}1$ such that~\eqref{eq:peluse-inverse} holds, since we may increase $B$ if necessary. 
 
 Suppose first that $j=0$. Using the notation~\eqref{eq:Lambda}, let
 \begin{align}\label{eq:etadef}
 \eta\coloneqq |\Lambda^{CN^d,N}_{P_1,\ldots, P_k}(1;f_0,\ldots, f_k)|.
 \end{align}
 Then $C\eta$ is equal to the left-hand side of~\eqref{eq:peluse-inverse}.
  We may clearly assume that $\eta\geq \delta$ and that $\delta\leq 1/L$ for any given constant $L=L_{C,d}$. We may further assume that  $N\geq \delta^{-K}\geq \eta^{-K}$ for any given constant $K=K_d$, since otherwise by taking $B=Kd_1$ the claim follows (with $q=1$ in~\eqref{eq:peluse-inverse}) from the crude triangle inequality bound 
  $$|\Lambda^{CN^d,N}_{P_1,\ldots, P_k}(1;f_0,\ldots, f_k)|\leq \frac{1}{CN^d}\sum_{m\in [-CN^d,CN^d]}|f_0(m)|.$$ 
  Now,  applying\footnote{In~~\cite[Theorem 3.3]{Peluse-FMP}, the functions $f_i$ are assumed to be supported on $[CN^d]$ instead of $[-CN^d,CN^d]$, but this makes no difference in the argument.}~\cite[Theorem 3.3]{Peluse-FMP} (with $\eta$ in place of $\delta$) we see that there exists some $1\leq B\ll_{d}1$ such that
 \begin{align*}
 \max_{\substack{q\in [\delta^{-B}]\\N'\in [\delta^{B}N^{d_1}, N^{d_1}]}}\frac{1}{N^{d}}\sum_{m\in \mathbb{Z}}\left|\frac{1}{N'}\sum_{n\in [N']}f_j(m+qn)\right|\gg_{C,d} \eta^{B},   \end{align*}
 which in view of~\eqref{eq:etadef} implies the claim.

 Suppose then that $j=1$. Then, making the change of variables $m'=m+P_1(n)$ in~\eqref{eq:Lambda}, we see that
 \begin{align*}
\Lambda^{CN^d,N}_{P_1,\ldots, P_k}(1;f_0,f_1,f_2,\ldots, f_k)=\Lambda^{CN^d,N}_{-P_1,P_2-P_1,\ldots, P_k-P_1}(1;f_1,f_0,f_2,\ldots, f_k)
 \end{align*}
 Now the claim follows from the $j=0$ case handled above. 
\end{proof}

\begin{proof}[Proof of Proposition~\ref{prop_loclin}]
We begin with a few reductions. Firstly, we may extend $\theta$ to a function on $\mathbb{Z}$ by setting it equal to $0$ outside $[N]$. Secondly, we may assume that $P_i(0)=0$ for $1\leq i\leq k$ by translating the functions $f_i$ if necessary. Thirdly, we may assume that $C$ is large enough in terms of $P_1,\ldots, P_k$ so that 
\begin{align*}
\max_{1\leq i\leq k}\max_{n\in [N]}|P_i(n)|\leq CN^d.    
\end{align*}  
Let 
\begin{align}\label{eq:etalambda}
\eta\coloneqq |\Lambda_{P_1,\ldots, P_k}^{CN^d,N}(\theta; f_0,f_1,\ldots, f_k)|.  
\end{align}
We may assume that $N\geq \eta^{-K}$ for any given constant $K$ depending on $d$, as otherwise the claim readily follows. 

We first wish to replace $f_0$ with a locally linear phase function.
For $m\in \mathbb{Z}$, define the first dual function 
\begin{align*}
F(m)\coloneqq \mathbb{E}_{n\in [N]}\theta(n)f_1(m+P_1(n))\cdots f_k(m+P_k(n)).
\end{align*}
Then by~\eqref{eq:etalambda} we have
\begin{align*}
\left|\sum_{m\in \mathbb{Z}}f_0(m)F(m)\right|\gg_C \eta N^d,   
\end{align*}
so by the Cauchy--Schwarz inequality we get
\begin{align}\label{eq:FF}
\left|\sum_{m\in \mathbb{Z}}\overline{F}(m)F(m)\right|\gg_C \eta^2N^d.     
\end{align}

By the definition of $F$,~\eqref{eq:FF} expands out as
\begin{align*}
\left|\sum_{m\in \mathbb{Z}}\mathbb{E}_{n\in [N]}\theta(n)\overline{F}(m)f_1(m+P_1(n))\cdots f_k(m+P_k(n))\right|\gg_C \eta^2 N^d.     
\end{align*}
Denote $d_1\coloneqq \deg P_1$. Applying Cauchy--Schwarz and van der Corput's inequality (Lemma~\ref{le:vdc}), this implies that
\begin{align*}
\left|\sum_{h\in [-N^{d_1},N^{d_1}]}\mu_{N^{d_1}}(h)\mathbb{E}_{n\in [N]}\sum_{m\in \mathbb{Z}}\Delta_h\overline{F}(m)\Delta_hf_1(m+P_1(n))\cdots \Delta_hf_k(m+P_k(n))\right|\gg_C \eta^4 N^d.   
\end{align*}
Noting that $N^{d_1}|\mu_{N^{d_1}}(h)|\ll 1_{|h|\leq N^{d_1}}$, from the triangle inequality and the pigeonhole principle we now see that 
\begin{align}\label{eq:Flarge}
\left|\mathbb{E}_{n\in [N]}\sum_{m\in \mathbb{Z}}\Delta_h\overline{F}(m)\Delta_hf_1(m+P_1(n))\cdots \Delta_hf_k(m+P_k(n))\right|\gg_C \eta^4 N^d     \end{align}
for $\gg_C \eta^{4}N^{d_1}$ integers $h\in [-N^{d_1},N^{d_1}]$.

Applying Lemma~\ref{le:degree-lower} and the pigeonhole principle, from~\eqref{eq:Flarge} we conclude that there exists a constant $1\leq B\ll_{d}1$ and an integer $1\leq q\ll_{C,P_1,\ldots, P_k} \eta^{-B}$ such that
\begin{align}\label{eq:qM}
 \mathbb{E}_{h\in [-N^{d_1},N^{d_1}]}\max_{M\in [\eta^{B}N^{d_1},N^{d_1}]}\mathbb{E}_{m\in [-2CN^d,2CN^d]}\left|\mathbb{E}_{y\in [qM]}1_{y\equiv 0\pmod q}\Delta_h F(m+y)\right|\gg_{C,P_1,\ldots, P_k}\eta^B.   
\end{align}
Using the simple bound 
\begin{align}\label{eq:shortening}
\mathbb{E}_{m\in [-X,X]}|\mathbb{E}_{y\in [Y]}a(m+y)|\ll \mathbb{E}_{m\in [-X,X]}|\mathbb{E}_{y\in [Y']}a(m+y)|+Y'/Y+Y/X,
\end{align}
valid for any bounded sequence $a\colon \mathbb{Z}\to \mathbb{C}$ and $1\leq Y'\leq Y\leq X$,
and setting 
$$N'=\eta^{4B+2}N^{d_1},$$
from~\eqref{eq:qM} we deduce that
\begin{align*}
 \mathbb{E}_{h\in [-N^{d_1},N^{d_1}]}\mathbb{E}_{m\in [-2CN^d,2CN^d]}\left|\mathbb{E}_{y\in [N']}1_{y\equiv 0\pmod q}\Delta_h F(m+y)\right|\gg_{C,P_1,\ldots, P_k}\eta^{B}.  
\end{align*}
Splitting the $h$ average into intervals of length $2N'$ and applying the pigeonhole principle, we see that there exists some integer $|\ell|\leq N^{d_1}$ such that
\begin{align*}
 \mathbb{E}_{h\in [-N',N']}\mathbb{E}_{m\in [-2CN^d,2CN^d]}\left|\mathbb{E}_{y\in [N']}1_{y\equiv 0\pmod q}\Delta_{h+\ell} F(m+y)\right|\gg_{C,P_1,\ldots, P_k}\eta^{B}.  
\end{align*}
From Cauchy--Schwarz and van der Corput's inequality, we then see that
\begin{align}\label{eq:eta10B}\begin{split}
&\mathbb{E}_{h\in [-N',N']}\mathbb{E}_{h'\in [-N',N']}\lfloor N'\rfloor \mu_{N'}(h')\mathbb{E}_{m\in [-2CN^d,2CN^d]}\mathbb{E}_{y\in [N']}1_{y,y+h'\equiv 0\pmod q}\Delta_{h+\ell} \Delta_{h'}F(m+y)\\
&\quad \gg_{C,P_1,\ldots, P_k} \eta^{2B}.
\end{split}
\end{align}

We wish to remove the weight $\mu_{N'}(h')$ from~\eqref{eq:eta10B}. To this end, we note the easily verified Fourier expansion
\begin{align*}
1-|x|=\frac{1}{2}+\frac{2}{\pi^2}\sum_{n\equiv 1\pmod 2}\frac{1}{n^2}e\left(\frac{n}{2}x\right)    
\end{align*}
for $x\in [-1,1]$, which allows us to write for $|h'|\leq N$ the expansion
\begin{align*}
\lfloor N'\rfloor \mu_{N'}(h')=1-\left|\frac{h'}{\lfloor N'\rfloor}\right| =\frac{1}{2}-\frac{4}{\pi^2}\sum_{n\equiv 1\pmod 2}\frac{1}{n^2}e\left(\frac{h'}{2\lfloor N'\rfloor}n\right).   
\end{align*}
Substituting this to~\eqref{eq:eta10B} and expanding $1_{y\equiv 0\pmod q}=\sum_{0\leq r<q}e(ry/q)$ and using $e(\xi h')=e(\xi(y+h+h'))e(-\xi(y+h))$ and $e(\xi y)=e(\xi(y+h))e(\xi(y+h'))e(-\xi(y+h+h'))$, we see that 
\begin{align*}
&\mathbb{E}_{h\in [-N',N']}\mathbb{E}_{h'\in [-N',N']}\mathbb{E}_{m\in [-2CN^d,2CN^d]}\mathbb{E}_{y\in [N']}\prod_{\omega\in \{0,1\}^2} F_{\omega}(m+y+\omega\cdot (h,h'))\\
&\quad \gg_{C,P_1,\ldots, P_k} \eta^{2B},
\end{align*}
where $F_{(0,0)}=F$ and $|F_{\omega}(x)|\ll_C 1$ for all $x, \omega$. 
Hence, by the Gowers--Cauchy--Schwarz inequality~\eqref{eq:gcs},
we find 
\begin{align*}
\mathbb{E}_{m\in [-2CN^d,2CN^d]}\|F\|_{U^2[m,m+N']}\gg_{C,P_1,\ldots, P_k} \eta^{2B}.
\end{align*}
Applying the pigeonhole principle, we deduce that
\begin{align}\label{eq:Mset}
\|F\|_{U^2[m,m+N']}\gg_{C,P_1,\ldots, P_k} \eta^{2B}
\end{align}
for $\gg \eta^{2B} N^d/N'$ integers $m\in N'\mathbb{Z}\cap [-CN^d,CN^d]$.

Note that for any complex number $z$ we have $|z|\leq 10 \max_{j\in \{0,1,2\}}\{\textnormal{Re}(e(j/3)z)\}$. For $j\in \{0,1,2\}$, let $\mathcal{M}_j$ be the set of $m\in N'\mathbb{Z}\cap [-CN^d,CN^d]$ for which
\begin{align}\label{eq:Mset2}
\sup_{\alpha\in [0,1]}\textnormal{Re}\left(e\left(\frac{j}{3}\right)\mathbb{E}_{x\in [m,m+N']}F(x)e(\alpha x)\right)\geq \eta^{4B+1}.    
\end{align}
Then, by the pigeonhole principle, the $U^2[N']$ inverse theorem~\eqref{eq:U2} and~\eqref{eq:Mset}, we have $|\mathcal{M}_{j_0}|\gg \eta^{2B} N^d/N'$ for some $j_0\in \{0,1,2\}$. For any $m\in \mathcal{M}_{j_0}$, let $\alpha_m'\in \frac{1}{N^{\deg P_1}}\mathbb{Z}$ be a point where the supremum in~\eqref{eq:Mset2} is attained, and let $\alpha_m$ be an element of $\frac{1}{N^{d_1}}\mathbb{Z}$ nearest to $\alpha_m'$. Recalling the definition of $N'$, we have 
\begin{align*}
\textnormal{Re}\left(e\left(\frac{j_0}{3}\right)\mathbb{E}_{x\in [m,m+N']}F(x)e(\alpha_m x)\right)\gg_{C,P_1,\ldots, P_k} \eta^{4B+1}.  \end{align*}

For $m\in N'\mathbb{Z}\setminus \mathcal{M}_{j_0}$, note that by Parseval's identity there is some $\alpha_m\in \frac{1}{N^{d_1}}\mathbb{Z}$ for which 
\begin{align*}
\left|\sum_{x\in [m,m+N']}F(x)e(\alpha_m x)\right| \ll_C (N')^{1/2}.   
\end{align*}
Now, extend the definition of $\alpha_m$ from $N'\mathbb{Z}$ to all of $\mathbb{Z}$ by letting $\alpha_m=\alpha_{m'}$, where $m'$ is the largest element of $N'\mathbb{Z}$ that is at most $m$. Then, define the locally linear phase function $\phi_0(m)=e(\alpha_mm)$, which has resolution $N'$. We now have
\begin{align*}
\left|\sum_{x\in [-CN^d,CN^d]}F(x)\phi_0(x)\right|\gg_{C,P_1,\ldots, P_k} \eta^{2B}\cdot \eta^{4B+1}
\end{align*}
Recalling the definition of $F$, we conclude that
\begin{align}\label{eq:Fconclusion}
\left|\sum_{m\in \mathbb{Z}}\sum_{n\in [N]}\theta(n)\phi_0(m)f_1(m+P_1(n))\cdots f_k(m+P_k(n))\right|\gg_{C,P_1,\ldots, P_k}\eta^{B'}N^{d+1},   
\end{align}
where $B'=6B+1$.

We proceed to replace also $f_1$ with a locally linear phase function.
For $m\in \mathbb{Z}$, define the second dual function as
\begin{align*}
G(m)\coloneqq \mathbb{E}_{n\in [N]} \theta(n)\phi_0(m-P_1(n))f_2(m+P_2(n)-P_1(n))\cdots f_k(m+P_k(n)-P_1(n)).   
\end{align*}
Making a change of variables, from~\eqref{eq:Fconclusion} it follows that
\begin{align*}
\left|\sum_{m\in \mathbb{Z}}f_1(m)G(m)\right|\gg_{C,P_1,\ldots, P_k}\eta^{B'}N^d.
\end{align*}
Arguing verbatim as above, we deduce that there exists a  locally linear phase function $\phi_1\:\mathbb{Z}\to \mathbb{C}$ of resolution $N''=\eta^{B''}N^{d_1}$ and with spectrum in $\frac{1}{N^{d_1}}\mathbb{Z}$ such that 
    \begin{align*}
\left|\sum_{m\in \mathbb{Z}}G(m)\phi_1(m)\right|\gg_{C,P_1,\ldots, P_k}\eta^{B''}N^d,     
\end{align*}
where $B''=6B+1'$.

Recalling the definition of $G$, this means that
\begin{align*}
\left|\sum_{m\in \mathbb{Z}}\sum_{n\in [N]}\theta(n)\phi_0(m)\phi_1(m+P_1(n))f_2(m+P_2(n))\cdots f_k(m+P_k(n))\right|\gg_{C,P_1,\ldots, P_k}\eta^{B''}N^{d+1}.    
\end{align*}
This gives the desired claim.
\end{proof}

\subsection{A circle method bound}\label{sub:circle}

The proof of Theorem~\ref{thm:GVNT2} will proceed by induction on $k$, so we first need to bound the weighted averages~\eqref{eq:Lambda} with $k=1$. These averages can be controlled simply by using classical Fourier analysis.

\begin{lemma}\label{le:circle1}
Let $d\in \mathbb{N}$ and $C\geq 1$. Let $P$ be a polynomial of degree $d$ with integer coefficients. Let $N\geq 1$, and let $f_0,f_1\colon \mathbb{Z}\to \mathbb{C}$ be functions supported on $[-CN^d,CN^d]$ with $|f_i|\leq 1$ for both $i\in \{0,1\}$, and let $\theta\colon [N]\to \mathbb{C}$ be a function. Then we have
\begin{align}\label{eq:circletheta}
\left|\frac{1}{N^{d+1}}   \sum_{m\in \mathbb{Z}}\sum_{n\in [N]}\theta(n)f_0(m)f_1(m+P(n))\right|\ll_{C,P} \|\theta\|_{u^{d+1}[N]}. 
\end{align}
\end{lemma}

\begin{proof}
Let $C'\geq 1$ (depending on $C,P$) be such that $\max_{n\in [N]}|P(n)|\leq (C'-C)N^d$. Then, by the orthogonality of characters, the left-hand side of~\eqref{eq:circletheta} without absolute values equals to
\begin{align*}
 &\frac{1}{N^{d+1}}\sum_{|a|\leq (C_P+C)N^d}\sum_{|m|\leq CN^d}\sum_{n\in [N]}\theta(n)f_0(m)f_1(a)1_{a=m+P(n)}\\
 &\quad =\frac{1}{N^{d+1}}\int_{0}^{1}\left(\sum_{|a|\leq C'N^d}f_1(a)e(\xi a)\right)\left(\sum_{|m|\leq CN^d}f_0(m)e(-\xi m)\right)\left(\sum_{n\in [N]}\theta(n)e(-\xi P(n))\right)\d \xi.   
\end{align*}
Now the claim follows by bounding the exponential sum involving $\theta$ pointwise by $N\|\theta\|_{u^{d+1}[N]}$ and by using Cauhcy--Schwarz and Parseval's identity to the remaining two exponential sums. 
\end{proof}

\subsection{Proof of Theorem~\ref{thm:GVNT2}}\label{sub:conclude}

We are now ready to prove the claimed estimate for the operator~\eqref{eq:Lambda} in the case of distinct degree polynomials.

\begin{proof}[Proof of Theorem~\ref{thm:GVNT2}] We use induction on $k$. The base case $k=1$ follows from Lemma~\ref{le:circle1}. Suppose that the case $k-1$ has been proven for some $k\geq 2$, and consider the case $k$. 

\textbf{Step 1: Reduction to locally linear phase functions.} Let $\delta\coloneqq |\Lambda_{P_1,\ldots, P_k}^{CN^d,N}(\theta; f_0,f_1,\ldots, f_k)|$. We may assume that $1/K''\geq \delta\geq N^{-1/K'}$ for any large constants $K'=K'_d$ and $K''=K''_{C,P_1,\ldots, P_k}$, as otherwise there is nothing to prove. By Proposition~\ref{prop_loclin}, there exist $C_d\geq 1$ and locally linear phase functions $\phi_0,\phi_1\colon\mathbb{Z}\to \mathbb{C}$ of resolution $\geq N'$ for some $N'\gg_{C,P_1,\ldots, P_k} \delta^{C_d}N^{\deg P_1}$, and with the spectra of $\phi_0,\phi_1$ belonging to $\frac{1}{N^{\deg P_1}}\mathbb{Z}$, such that
\begin{align}\label{eq:deltalambda0}
|\Lambda_{P_1,\ldots, P_k}^{CN^d,N}(\theta; \phi_0,\phi_1,f_2,\ldots, f_k)|\geq \delta^{C_d}.   \end{align}

\textbf{Step 2: An iteration for the locally linear phase function.}
We can write $\phi_1(m)=e(\alpha_m m)$ for some $\alpha_m\in \frac{1}{N^{\deg P_1}}\mathbb{Z}$, with $m\mapsto \alpha_m$ is being constant on the the intervals $[jN'+a,(j+1)N'+a)$ for some $a\in [N']$ and all $j\in \mathbb{Z}$. For any set $S\subset \mathbb{R}$, write $\phi_{1,S}(m)\coloneqq \phi_1(m)1_{\alpha_m\not \in S}$. 

\textbf{Claim.} If $C_d'$ and $K'_d$ are large, the following holds.  For any $\eta \in (0,N^{-1/K_d'})$ and any finite (possibly empty) set $S$, if
\begin{align}\label{eq:deltalambda}
|\Lambda_{P_1,\ldots, P_k}^{CN^d,N}(\theta; \phi_0,\phi_{1,S},f_2,\ldots, f_k)|\geq \eta,
\end{align}
then either 
\begin{align}\label{eq:deltalambda3}
|\Lambda_{P_1,\ldots, P_k}^{CN^d,N}(\theta; \phi_{0},\phi_1-\phi_{1,\{\alpha\}},f_2,\ldots, f_k)|\geq \eta^{2C_d'}   \end{align}
or there exists $\alpha\not \in S$ such that $\alpha_m=\alpha$ for $\geq \eta^{C_d'}N^d$ integers $m\in [-CN^d,CN^d]$ and  
\begin{align}\label{eq:deltalambda2}
|\Lambda_{P_1,\ldots, P_k}^{CN^d,N}(\theta; \phi_{0},\phi_{1,S\cup \{\alpha\}},f_2,\ldots, f_k)|\geq \eta-\eta^{2C_d'}.
\end{align}

For proving this claim, we first apply van der Corput's inequality (Lemma~\eqref{le:vdc}) to~\eqref{eq:deltalambda} to conclude that there is a set $\mathcal{H}\subset [-CN^d,CN^d]$ of size $\gg \eta^2 N^d$ such that for $h\in \mathcal{H}$ we have
\begin{align*}
 |\Lambda_{P_1,\ldots, P_k}^{CN^d,N}(1; \Delta_h\phi_0,\Delta_h\phi_{1,S},\Delta_hf_2,\ldots, \Delta_hf_k)|\gg \eta^2.  
\end{align*}
From Lemma~\ref{le:degree-lower} and the pigeonhole principle, we now conclude that for some constant $B_d\geq 1$, some integer $1\leq q\leq \eta^{-B_d}$ and some $M\in [\eta^{B_d}N^{\deg P_1},N^{\deg P_1}]$, we have
\begin{align*}
\mathbb{E}_{|x|\leq 2CN^d}|\mathbb{E}_{m\in [M]}\overline{\phi_{1,S}}(x+qm+h)\phi_{1,S}(x+qm)|\gg_{C,P_1,\ldots, P_k} \eta^{B_d}
\end{align*}
for $\gg_{C,P_1,\ldots, P_k}  \eta^{B_d} N^d$ integers $h\in \mathcal{H}$. Let $\mathcal{H}'$ be the set of such $h$. By~\eqref{eq:shortening} we also have 
\begin{align}\label{eq:phi1S}
\mathbb{E}_{|x|\leq 2CN^d}|\mathbb{E}_{m\in [M']}\overline{\phi_{1,S}}(x+qm+h)\phi_{1,S}(x+qm)|\gg_{C,P_1,\ldots, P_k} \eta^{B_d}
\end{align}
for $h\in \mathcal{H}'$, where $M'=\eta^{4B_d+2}N'$.

From the pigeonhole principle, we see that there exist $h_0\in [N']$ and $\mathcal{H}''\subset \mathcal{H}'$ with $|\mathcal{H}''|\gg_{C,P_1,\ldots, P_k}\eta^{4B_d}N^d$ such that $\|(h-h_0)/N'\|\leq \eta^{3B_d}$ for all $h\in \mathcal{H}'$. Let $\mathcal{X}$ be the set of $x\in [-CN^d,CN^d]\cap \mathbb{Z}$ for which
\begin{align*}
\min_{h'\in \{0,h_0\}}\left\|\frac{x+h'-a}{N'}\right\|\geq \eta^{2B_d}.    
\end{align*}
Then $|\mathcal{X}|\geq (2C-O(\eta^{2B_d}))N^d$.
Note that, for any $x\in \mathcal{X}$, $h\in \mathcal{H}''$ and $m\in [\eta^{2B_d}N'/2]$, we have
\begin{align*}
\overline{\phi_{1,S}}(x+m+h)\phi_{1,S}(x+m)=1_{\alpha_{x}\not \in S}\cdot 1_{\alpha_{x+h}\not \in S}e((\alpha_x-\alpha_{x+h})(x+m)-h\alpha_{x+h}). 
\end{align*}
Hence, recalling~\eqref{eq:phi1S} and applying the pigeonhole principle, we see that there is some $x_0\in \mathcal{X}$ with $\alpha_{x_0}\not \in S$ such that for all $h\in \mathcal{H}''$ we have $\alpha_{x_0+h}\not \in S$ and 
\begin{align*}
|\mathbb{E}_{m\in [M']}e((\alpha_{x_0}-\alpha_{x_0+h})qm)|\gg_{C,P_1,\ldots, P_k} \eta^{B_d}.
\end{align*}
By the geometric sum formula, we conclude that, for all $h\in \mathcal{H}''$, we have 
\begin{align*}
\|q(\alpha_{x_0}-\alpha_{x_0+h})\|\ll_{C,P_1,\ldots, P_k} \frac{\eta^{-B_d}}{N^{\deg P_1}}.    
\end{align*}
Recalling that $\alpha_m\in \frac{1}{N^{\deg P_1}}\mathbb{Z}$, by the pigeonhole principle we conclude that there is some constant $C_d'\geq 1$ and some $\alpha\in \frac{1}{N^{\deg P_1}}\mathbb{Z}$ such that $\alpha_m=\alpha$ for $\geq \eta^{C_d'}N^{d}$ integers $m\in [-CN^d,CN^d]$. We have $\alpha\not \in S$, since $\alpha_{x_0+h}\not \in S$ for $h\in \mathcal{H}''$. Now the claim follows by writing $\phi_{1,S}=(\phi_1-\phi_{1,\{\alpha\}})+\phi_{1,S\cup\{\alpha\}}$ and applying the pigeonhole principle. 

\textbf{Step 3: Concluding the argument.}
Now, applying repeatedly the claim established above, starting with $S=\emptyset$ (in which case the assumption~\eqref{eq:deltalambda} holds for $\eta=\delta^{C_d'}$ by \eqref{eq:deltalambda0}) and applying the above repeatedly, after $\ll \delta^{-C_d'}$ iterations~\eqref{eq:deltalambda2} cannot hold (since the number of different values that $\alpha_m$ takes at least $(\delta/2)^{C_d'}N^d$ times is $\ll_d \delta^{-C_d'}C$), so~\eqref{eq:deltalambda3} holds with $\eta=\delta^{A_d}$ for some constant $A_d$. 

 Now that~\eqref{eq:deltalambda3} holds with $\eta=\delta^{A_d}$, by the pigeonhole principle there exist intervals $[N_1,N_2]\subset [1,N]$ of length $\delta^{3A_dC_d'}N$ and $I\subset [-CN^d,CN^d]$ of length $\delta^{3A_dC_d'}N^d$ such that, denoting $\psi_{\alpha}(m)\coloneqq 1_{\alpha_m=\alpha}$, we have
\begin{align*}
|\Lambda_{P_1,\ldots, P_k}^{CN^d,N}(\theta e(\alpha P_1(\cdot))1_{[N_1,N_2]}; \phi_0e(\alpha \cdot)1_{I}, \psi_{\alpha},f_3,\ldots, f_k)|\geq \delta^{8A_dC_d'}.   
\end{align*}
But since the function $\psi_{\alpha}$ is constant on intervals of the form $[jN'+a,(j+1)N'+a)$ with $j\in \mathbb{Z}$, this implies
\begin{align*}
|\Lambda_{P_1,\ldots, P_k}^{CN^d,N}(\theta e(\alpha P_1(\cdot))1_{[N_1,N_2]}; \phi_0e(\alpha \cdot)1_{I}\psi_{\alpha}(\cdot+P_1(N_1)),f_3,\ldots, f_k)|\gg\delta^{8A_dC_d'}.   
\end{align*}
By the induction assumption and the fact that $d>\deg P_1$, for some $A_d'$ we now obtain
\begin{align*}
\|\theta 1_{[N_1,N_2]}\|_{u^{d+1}[N]}\gg_{C,P_1,\ldots, P_k} \delta^{A_d'}.
\end{align*}
By Vinogradov's Fourier expansion (Lemma~\ref{le:vinogradov}), we can write 
$$1_{[N_1,N_2]}(n)=\sum_{1\leq j\leq \delta^{-3A_d'}}c_je(\beta_j n)+E(n)$$
for some real numbers $\beta_j$, some complex numbers $|c_j|\leq 1$ and some $|E(n)|\ll 1$ satisfying $\sum_{n\in [N]}|E(n)|\ll \delta^{2A_d'}N$. Hence, we conclude that 
\begin{align*}
\|\theta\|_{u^{d+1}[N]}\gg_{C,P_1,\ldots, P_k} \delta^{O_{C,P_1,\ldots, P_k}(1)},
\end{align*}
as desired. 
\end{proof}

\subsection{Proof of Theorem~\ref{thm:GVNT1}}\label{sub:genpoly}

For the proof of Theorem~\ref{thm:GVNT1}, we need the following generalised von Neumann theorem for arithmetic progressions. This is well known (see for example~\cite[Lemma 2]{FHK}), although the result is typically presented for functions defined on a cyclic group.

\begin{lemma}[A generalised von Neumann theorem for arithmetic progressions]\label{le:GVNT0}

Let $k\in \mathbb{N}$, $N\geq 1$, $C\geq 1$, and let $\theta\:[N]\to \mathbb{C}$. Let $L_1,\ldots, L_k$ be polynomials with integer coefficients of degree at most $1$ satisfying $L_i([N])\subset [-CN,CN]$ for all $i\in [k]$. Let $f_1\ldots, f_k\colon \mathbb{Z}\to \mathbb{C}$ be functions supported on $[-CN,CN]$ with $|f_i|\leq 1$ for $1\leq i\leq k$. Then we have
\begin{align*}
\left|\frac{1}{N^2}\sum_{m\in \mathbb{Z}}\sum_{n\in [N]}\theta(n)f_1(m+L_1(n))\cdots f_k(m+L_k(n))\right|\ll_{k,C}\|\theta\|_{U^{k}[N]}.     
\end{align*}
\end{lemma}

Lemma~\ref{le:GVNT0} could be proven directly without difficulty, but we deduce it as an immediate consequence of the following more general lemma that deals with multidimensional averages. This multidimensional version is needed for proving Theorem~\ref{thm_commuting} (but for Theorem~\ref{thm:GVNT1} the one-dimensional case suffices).

\begin{lemma}[A multidimensional generalised von Neumann theorem for arithmetic progressions]\label{le:GVNT-multi}
 Let $k,r\in \mathbb{N}$, $N\geq 1$, $C\geq 1$, and let $\theta\:[N]\to \mathbb{C}$. Let $L_1,\ldots, L_k\colon \mathbb{Z}\to \mathbb{Z}^r$ be polynomials with integer coefficients of degree at most $1$ satisfying $L_i([N])\subset [-CN,CN]^r$ for all $i\in [k]$. Let $f_1\ldots, f_k\colon \mathbb{Z}^r\to \mathbb{C}$ be functions supported on $[-CN,CN]^r$ with $|f_i|\leq 1$ for $1\leq i\leq k$. Then we have
\begin{align}\label{eq:GVNT0-1}\begin{split}
&\left|\frac{1}{N^{r+1}}\sum_{\mathbf{m}\in \mathbb{Z}^r}\sum_{n\in [N]}\theta(n)f_1(\mathbf{m}+L_1(n))\cdots f_k(\mathbf{m}+L_k(n))\right|\\
&\quad \ll_{k,r,C}\|\theta\|_{U^{k}[N]}.     
\end{split}
\end{align}   
\end{lemma}

\begin{proof}[Proof of Lemma~\ref{le:GVNT-multi}]
For convenience, we extend the definition of $\theta$ to all of $\mathbb{Z}$ by setting it equal to $0$ outside $[N]$. 

We use induction on $k$. In the case $k=1$, the claim is immediate by making the change of variables $\mathbf{m}'=\mathbf{m}+n\mathbf{v}_1+\mathbf{b}_1$ and noting that $\mathbb{E}_{n\in [N]}\theta(n)=\|\theta\|_{U^1[N]}$. 

Suppose then that the claim holds in the case $k-1\in \mathbb{N}$ and consider the case $k$. Let $S$ be the expression inside the absolute values in~\eqref{eq:GVNT0-1}. By making the change of variables $\mathbf{m}'=\mathbf{m}+L_1(n)$, we have
\begin{align*}
S=\frac{1}{N^{r+1}}\sum_{\mathbf{m}'\in \mathbb{Z}^r}\sum_{n\in [N]}\theta(n)f_1(\mathbf{m}')\prod_{j=2}^k f_j(\mathbf{m}'+L_j(n)-L_1(n)).    
\end{align*}
By the Cauchy--Schwarz inequality, we obtain
\begin{align*}
|S|^2\ll_{r,C} \frac{1}{N^{r+2}}\sum_{\mathbf{m}\in \mathbb{Z}^{r}}\left|\sum_{n\in [N]}\theta(n)\prod_{j=2}^k f_j(\mathbf{m}+L_j(n)-L_1(n))\right|^2.     
\end{align*}

In what follows, for a function $f\colon \mathbb{Z}^r\to \mathbb{C}$ and $\mathbf{v}\in \mathbb{Z}^r$, denote $\Delta_{h}^\mathbf{v}f(x)\coloneqq \overline{f(x+h\mathbf{v})}f(x)$. 
From van der Corput's inequality (Lemma~\ref{le:vdc}), we conclude that
\begin{align*}
|S|^2\ll_{r,C} \frac{1}{N^{r+1}}\sum_{h\in \mathbb{Z}}\mu_N(h)\sum_{\mathbf{m}\in \mathbb{Z}^r}\sum_{n\in [N]}\Delta_h\theta(n)\prod_{j=2}^k \Delta_h^{\mathbf{v}_j}f_j(\mathbf{m}+L_j(n)-L_1(n)),\end{align*}
where $\mathbf{v}_j\in \mathbb{Z}^r$ is such that $L_j(n)-L_1(n)=n\mathbf{v}_j+L_j(0)-L_1(0)$.
By the induction assumption and the fact that $|\mu_N(h)|\ll \frac{1}{N}1_{|h|\leq N}$, we see that
\begin{align}\label{eq:Deltah}
|S|^2\ll_{k,r, C} \mathbb{E}_{h\in [-N,N]}\|\Delta_h \theta\|_{U^{k-1}[N]}.   
\end{align}
But by H\"older's inequality and the definition of the $U^{k-1}[N]$ norm, the right-hand side of~\eqref{eq:Deltah} is
\begin{align*}
&\ll_{k,C} \left(\mathbb{E}_{h\in [-N,N]}\mathbb{E}_{n,h_1,\ldots, h_{k-1}\in [-N,N]}\prod_{\omega\in \{0,1\}^{k-1}}\mathcal{C}^{|\omega|}\Delta_h\theta(n+\omega\cdot (h_1,\ldots, h_{k-1}))\right)^{1/2^{k-1}}\\
&\ll_k\|\theta\|_{U^k[N]}^2. 
\end{align*}
This completes the induction.
\end{proof}

\begin{proof}[Proof of Theorem~\ref{thm:GVNT1}]

For convenience, we extend the definition of $\theta$ to all of $\mathbb{Z}$ by setting it equal to $0$ outside $[N]$. 

We apply the PET induction scheme of Bergelson and Leibman~\cite{bl}. For any finite collection $\mathcal{Q}$ of polynomials, define its \emph{type} as $(d,w_d,\ldots, w_1)$, where $w_j$ is the number of different leading coefficients among the polynomials in the subcollection $\{Q\in \mathcal{Q}\: \deg Q=j\}$. We introduce the lexicographic order $<$ on the types of polynomials. In other words, we write $(d,w_d,\ldots, w_1)<(d',w_{d'}',\ldots, w_{1}')$ if there exists $j\geq 0$ such that the first $j$ coordinates of the vectors are equal and the coordinate of order $j+1$ is larger for the second vector. In this way, we have introduced an order $<$ on the set of all finite collections of polynomials based on the order of their type. Note that the length of any descending chain of collections of polynomials with maximal element $\mathcal{Q}$ is bounded as a function of $|\mathcal{Q}|$ and $\deg \mathcal{Q}$.

We shall prove Theorem~\ref{thm:GVNT1} by induction on the type of  $\mathcal{Q}$.
The base case is that of collections of type $(1,k)$ with $k\in \mathbb{N}$, that is, collections where all the polynomials have degree at most $1$. This case follows readily from Lemma~\ref{le:GVNT0}.  

Suppose that $\mathcal{Q}$ is a finite collection of polynomials for which Theorem~\ref{thm:GVNT1} fails and that there is no smaller collection $\mathcal{Q}'<\mathcal{Q}$ with this property. Then $\deg \mathcal{Q}\geq 2$. Let us write $\mathcal{Q}=\{Q_1,\ldots, Q_k\}$, where $Q_1$ is one of the polynomials of $\mathcal{Q}$ with the least positive degree.

Let $S$ be the left-hand side of~\eqref{eq:GVNT1-1poly}. By the Cauchy--Schwarz inequality, van der Corput's inequality~\eqref{eq:vdc} and a change of variables, we have
\begin{align}\label{eq:PET}\begin{split}
|S|^2&\ll \frac{1}{N^{d+1}}\sum_{h\in \mathbb{Z}}\mu_N(h)\sum_{m\in \mathbb{Z}}\sum_{n\in [N]}\Delta_h\theta(n)\prod_{Q\in \mathcal{Q}}\Delta_h f_Q(m+Q(n))\\
&=\frac{1}{N^{d+1}}\sum_{h\in \mathbb{Z}}\mu_N(h)\sum_{m\in \mathbb{Z}}\sum_{n\in [N]}\Delta_h\theta(n)\prod_{Q\in \mathcal{Q}}f_{Q}(m+Q(n+h)-Q_1(n))\overline{f_Q}(m+Q(n)-Q_1(n))\\
&=\frac{1}{N^{d+1}}\sum_{h\in \mathbb{Z}}\mu_N(h)\sum_{m\in \mathbb{Z}}\sum_{n\in [N]}\Delta_h\theta(n)\prod_{\widetilde{Q}\in \mathcal{Q}_h}\widetilde{f}_{\widetilde{Q}}(m+\widetilde{Q}(n)),\end{split}
\end{align}
where $\widetilde{f}_{\widetilde{Q}}$ are some $1$-bounded functions and $\mathcal{Q}_h$ is the collection (of size $\leq 2|\mathcal{Q}|$ and degree $\leq \deg \mathcal{Q}$) given by
\begin{align*}
\mathcal{Q}_h=\{Q_j(y+h)-Q_1(y), Q_j(y)-Q_1(y)\}_{j=1}^{k}.   \end{align*}
We claim that for every $h\in \mathbb{Z}$ the type of $\mathcal{Q}_h$ is less than the type of $\mathcal{Q}$. Let $d_1=\deg Q_1$. Let $(d,w_d,\ldots, w_1)$ be the type of $\mathcal{Q}$ and let $(d_h,w_{d,h},\ldots, w_{1,h})$ be the type of $\mathcal{Q}_h$. We have $d_h\leq d$, and if $d_h=d$, then $w_{e}=w_{e,h}$ for $d_1<e\leq d$. Moreover, $w_{d_1,h}<w_{d_1}$, since if $c_1,\ldots, c_r$ are the distinct leading coefficients of the degree $d_1$ polynomials in $\mathcal{Q}$, the leading coefficients of degree $d_1$ polynomials in $\mathcal{Q}_h$ are $c_2-c_1,\ldots, c_r-c_1$. Hence we have $\mathcal{Q}_h<\mathcal{Q}$.

By~\eqref{eq:PET}, the estimate $|\mu_N(h)|\ll \frac{1}{N}1_{|h|\leq N}$  and the pigeonhole principle, we have
\begin{align*}
|S|^2&\ll \mathbb{E}_{h\in [-N,N]}\max_{|j|\leq (C+1)N^{d-d_h}}\left|\frac{1}{N^{d_h+1}}\sum_{m\in \mathbb{Z}}1_{(jN^{d_h},(j+1)N^{d_h}]}(m)\sum_{n\in [N]}\Delta_h\theta(n)\prod_{\widetilde{Q}\in \mathcal{Q}_h}\widetilde{f}_{\widetilde{Q}}(m+\widetilde{Q}(n))\right|\\
&= \mathbb{E}_{h\in [-N,N]}\max_{|j|\leq (C+1)N^{d-d_h}}\left|\frac{1}{N^{d_h+1}}\sum_{m'\in \mathbb{Z}}\sum_{n\in [N]}\Delta_h\theta(n)1_{[N^{d_h}]}(m')\prod_{\widetilde{Q}\in \mathcal{Q}_h}\widetilde{f}_{\widetilde{Q},jN^{d_h}}(m'+\widetilde{Q}(n))\right|,
\end{align*}
where $\widetilde{f}_{\widetilde{Q},a}=\widetilde{f}_{\widetilde{Q}}(\cdot+a)$. 

Since by assumption $Q([N])\subset [-CN^d,CN^d]$, basic linear algebra gives that the coefficients of $Q$ are $\ll_{C,d}1$ in modulus. Then, by the mean value theorem, for any $\widetilde{Q}\in \mathcal{Q}_h$ we have 
$$\max_{n\in [N]}|\widetilde{Q}(n)|\ll |\widetilde{Q}(1)|+N\max_{y\in [1,N]}|\widetilde{Q}'(y)|\ll_{\deg \mathcal{Q},C}N^{d_h},$$
so for all $n\in [N]$ we have $\widetilde{Q}(n)\in [-C'N^{d_h},C'N^{d_h}]$ for some $C'\ll_{\deg \mathcal{Q},C}1$. Now, by the induction assumption (and the fact that the functions $\widetilde{f}_{Q,jN^d_h}$ may be assumed to be supported on $[-(C'+1)N^{d_h},(C'+1)N^{d_h}]$), we conclude that for some natural number $s'\ll_{|\mathcal{Q}|,\deg{Q}}$ we have 
\begin{align*}
|S|^2\ll_{|\mathcal{Q}|, \deg \mathcal{Q},C} \mathbb{E}_{h\in [-N,N]}\|\Delta_h \theta\|_{U^{s'}[N]}. 
\end{align*}
Applying H\"older's inequality as in the proof of Lemma~\ref{le:GVNT0}, this implies that
\begin{align*}
|S|\ll_{|\mathcal{Q}|,\deg{Q},C} \|\theta\|_{U^{s'+1}[N]}.  \end{align*}
This completes the induction.
\end{proof}

\section{Quantitative uniformity of multiplicative functions}\label{sec:lemmas}

In this section, we give quantitative bounds for the uniformity norms of the M\"obius function and other multiplicative functions that we need for the proofs of the main theorems.

\subsection{The M\"obius function}

The only arithmetic property we need of the M\"obius function is encapsulated in the following recent result of Leng~\cite{lengII}, building on~\cite{LSS} and improving on~\cite{TT-JEMS}.

\begin{lemma}[Quantitative $U^k$-uniformity of $\mu$]\label{le:leng} Let $k\in \mathbb{N}$, $A>0$ and $N\geq 3$. Then we have
    \begin{align*}
     \|\mu\|_{U^k[N]}\ll_A (\log N)^{-A}.   
    \end{align*}
\end{lemma}

\begin{proof}
This follows by combining Leng's result~\cite[Theorems 6]{lengII} with \cite[Theorem 2.5]{TT-JEMS} (where we can use Siegel's bound $q_{\textnormal{Siegel}}\gg_A (\log N)^{A}$) and the triangle inequality for the Gowers norms. 
\end{proof}

\begin{remark}\label{rmk:vin} We mention that the weaker bound $\|\mu\|_{u^k[N]}\ll_{A} (\log N)^{-A}$, which turns out to be all that we need for the proof of Theorem~\ref{thm_polyergodic} in the case of distinct degree polynomials, is much simpler to prove. Indeed, it follows from the method to bilinear exponential sums developed by Vinogradov~\cite{vinogradov}.  
\end{remark}

\begin{remark}\label{rmk-liouville}
Lemma~\ref{le:leng} continues to hold if the M\"obius function $\mu$ is replaced with the Liouville function $\lambda$. This follows easily from the identities
\begin{align*}
\lambda(n)=\sum_{d\geq 1}\mu(n/d)1_{d^2\mid n},\quad  \textnormal{and}\quad \mu(n/d)1_{d^2\mid n}=\sum_{\substack{n=en'\\d^2\mid e\mid d^{\infty}\\(d,n')=1}}\mu(e/d)\mu(n'),  
\end{align*}
which can be truncated to $d^2\leq e\leq K$ for any $K\geq 1$ at the cost of an error term that is bounded by $O(1/K)$ in $L^1[N]$ norm. 
\end{remark}

\subsection{Multiplicative functions satisfying the Siegel--Walfisz condition}\label{sec:Siegel}

Our goal in this subsection is to show that any multiplicative function satisfying the Siegel--Walfisz assumption (Definition~\ref{def:SW}) is close to a function whose $u^k[N]$ norm decays faster than any power of logarithm.

\begin{proposition}\label{prop:SW}
Let $k\in \mathbb{N}$ and $A>0$. Let $g\colon \mathbb{N}\to \mathbb{C}$ be a multiplicative function satisfying the Siegel--Walfisz property. Then we have a decomposition $g=g_1+g_2$ with $g_1,g_2\colon \mathbb{N}\to \mathbb{C}$ satisfying $|g_1|, |g_2|\leq |g|$ and such that for $N\geq 3$ we have
\begin{align*}
\|g_1\|_{u^k[N]}\ll_{A,k} (\log N)^{-A}.      
\end{align*}
and for $r\geq 1$ we have
\begin{align*}
\mathbb{E}_{n\leq N}|g_2(n)|^{r}\ll_r (\log N)^{-1+o(1)}.    
\end{align*}
\end{proposition}

Throughout this section, let
\begin{align*}
Q_n\coloneqq\exp((\log \log n)^2),\quad R_n\coloneqq \exp((\log n)/(\log \log n)^2)    
\end{align*}
for $n\geq 3$ and $Q_1=Q_2=R_1=R_2=1$, and define the function
\begin{align}\label{eq:gtilde}
\widetilde{g}(n)\coloneqq g(n)(1-1_{p\mid n\implies p\not \in (Q_n,R_n)}).
\end{align}
In other words, $\widetilde{g}$ is the restriction of $g$ to those integers having at least one prime factor from $(Q_n,R_n)$. The following lemma shows that the function $\widetilde{g}$ is close to $g$ in $L^r[N]$ norm.

\begin{lemma}\label{le:Lr}
Let $C>0$, $r\geq 1$, and let $g\colon \mathbb{N}\to \mathbb{C}$ satisfy $|g(n)|\leq d(n)^{C}$ for all $n\in \mathbb{N}$. Then for $N\geq 3$ we have
\begin{align}\label{eq:f-f}
\mathbb{E}_{n\leq N}|g(n)-\widetilde{g}(n)|^r\ll_{C,r} (\log N)^{-1+o(1)}. \end{align}
\end{lemma}

\begin{proof} The left-hand side of~\eqref{eq:f-f} is by the triangle inequality and the upper bound on $g$ bounded by
\begin{align*}
\ll \mathbb{E}_{n\leq N}d(n)^{Cr}1_{p\mid n\implies p\not \in (Q_N,R_{\sqrt{N}})}+O(N^{-1/2+o(1)})\coloneqq S+O(N^{-1/2+o(1)}).    
\end{align*}
Applying Shiu's bound~\cite[Theorem 1]{shiu} followed by Mertens's theorem, we see that 
\begin{align*}
S&\ll \prod_{p\in [1,N]\setminus(Q_N, R_{\sqrt{N}})}\left(1+\frac{d(p)^{Cr}-1}{p}\right)\prod_{p\in (Q_N, R_{\sqrt{N}})}\left(1-\frac{1}{p}\right)\\
&= \prod_{p\leq N}\left(1+\frac{2^{Cr}-1}{p}\right) \prod_{p\in (Q_N, R_{\sqrt{N}})}\left(1+\frac{2^{Cr}-1}{p}\right)^{-1}\left(1-\frac{1}{p}\right)\\
&\ll_{C,r} (\log N)^{2^{Cr}-1}\left(\frac{\log Q_N}{\log R_{\sqrt{N}}}\right)^{2^{Cr}}\\
&\ll_{C,r} (\log N)^{-1+o(1)},   
\end{align*}
as desired.
\end{proof}

The next lemma shows that the condition on the prime factors in the definition of $\widetilde{g}$ can be replaced with a sieve weight up to a small error. 

\begin{lemma}\label{le:fund} Let $N\geq 3$ and $C>0$. There exist real numbers $\lambda_r\in [-1,1]$ such that for any $A>0$ we have 
\begin{align}\label{eq:lambdad}
\mathbb{E}_{N-N(\log N)^{-A}<n\leq N}\left|1_{p\mid n\implies p\not \in (Q_n,R_n)}-\sum_{\substack{r\mid n\\r\leq N^{1/10}}}\lambda_r\right|\cdot d(n)^{C}\ll_{A,C} (\log N)^{-A/2}.    
\end{align}
\end{lemma}

\begin{proof}
By splitting intervals into shorter ones if necessary, we may assume that $A$ is large enough in terms of $C$. Let $\lambda_r$ be the upper bound linear sieve coefficients of level $N^{1/10}$ and sifting range $(Q_N,R_{N-N(\log N)^{-A}})$, as defined in~\cite[Section 12]{opera}. From the definition we have $\lambda_r\in \{-1,0,+1\}$. Let $\nu(n)=\sum_{r\mid n}\lambda_r$. Then the left-hand side of~\eqref{eq:lambdad} is
\begin{align*}
&\ll (\log N)^{A/2} \mathbb{E}_{N-N(\log N)^{-A}<n\leq N}|\nu(n)-1_{p\mid n\implies p\not \in (Q_n,R_n)}|\\
&\quad+\mathbb{E}_{N-N(\log N)^{-A}<n\leq N}d(n)^{C+1}1_{d(n)\geq (\log N)^{A/C}}\\
&\coloneqq S_1+S_2.
\end{align*}
Using Shiu's bound~\cite[Theorem 1]{shiu}, we see that
\begin{align*}
S_2\ll (\log N)^{-2A} \mathbb{E}_{N-N(\log N)^{-A}<n\leq N}d(n)^{2C+1}\ll (\log N)^{-2A+2^{2C+1}-1}\ll (\log N)^{-A}   
\end{align*}
by the assumption that $A$ is large in terms of $C$. 

Since $$\nu(n)\geq 1_{p\mid n\implies p\not \in (Q_N,R_{N-N(\log N)^{-A}})}\geq 1_{p\mid n\implies p\not\in (Q_n,R_n)},$$
for $n\in (N-N(\log N)^{-A},N]$, we have
\begin{align*}
&\mathbb{E}_{N-N(\log N)^{-A}<n\leq N}|\nu(n)-1_{p\mid n\implies p\not \in (Q_n,R_n)}|\\
&\quad = \mathbb{E}_{N-N(\log N)^{-A}<n\leq N}(\nu(n)-1_{p\mid n\implies p\not \in (Q_N,R_{N-N(\log N)^{-A}})})\\
&\quad \quad+O\left(\sum_{p\in (Q_{N-N(\log N)^{-A}},Q_N)\cup (R_{N-N(\log N)^{-A}},R_N)}\frac{1}{p^2}\right).     
\end{align*}
By the fundamental lemma of sieve theory (\cite[Lemma 6.8]{iw-kow}) and Mertens's theorem, this is
\begin{align*}
\ll \exp(-(\log \log N)^2/20)+(\log N)^{-A/2},     
\end{align*}
which suffices.
\end{proof}

We are now ready to show that the Siegel--Walfisz property for $g$ implies the same property for $\widetilde{g}$. 

\begin{lemma}\label{le:SW}
 Let $g\colon \mathbb{N}\to \mathbb{C}$ be a multiplicative function satisfying the Siegel--Walfisz property. Then $\widetilde{g}$ satisfies the Siegel--Walfisz property.  
\end{lemma}

\begin{proof}
Since $\widetilde{g}(n)=g(n)-g(n)1_{p\mid n\implies p\not \in (Q_n,R_n)}$, it suffices to show that the function
\begin{align*}
n\mapsto g(n)1_{p\mid n\implies p\not \in (Q_n,R_n)}    
\end{align*}
satisfies the Siegel--Walfisz property. By splitting a long interval into shorter ones, it suffices to show that for any large $A>0$ and any $1\leq a\leq q\leq (\log N)^{A}$ we have 
\begin{align*}
\left|\mathbb{E}_{N-N(\log N)^{-A}<n\leq N}g(n)1_{p\mid n\implies p\not \in (Q_n,R_n)} 1_{n\equiv a\pmod q}\right|\ll_A (\log N)^{-A/2}.     
\end{align*}
By Lemma~\ref{le:fund}, it suffices to show that for any $|\lambda_r|\leq 1$ we have
\begin{align*}
\left|\mathbb{E}_{N-N(\log N)^{-A}<n\leq N}g(n)\sum_{\substack{r\mid n\\r\leq N^{1/10}}}\lambda_r 1_{n\equiv a\pmod q}\right|\ll_A (\log N)^{-A/2}.     
\end{align*}
Exchanging the order of summation and applying the triangle inequality, it suffices to show that
\begin{align*}
\sum_{r\leq N^{1/10}}\frac{1}{r}\left|\mathbb{E}_{(N-N(\log N)^{-A})/r<m\leq N/r}g(rm)1_{rm\equiv a\pmod q}\right|\ll_A (\log N)^{-A/2}.  
\end{align*}

Let $\mathcal{A}_N$ be the set of $r\in \mathbb{N}$ with $\omega(r)\leq (\log \log N)^{3/2}$. Then by Shiu's bound we have
\begin{align*}
&\sum_{\substack{r\leq N^{1/10}\\r\not \in \mathcal{A}_N}}\frac{1}{r}\left|\mathbb{E}_{(N-N(\log N)^{-A})/r<m\leq N/r}g(rm)1_{rm\equiv a\pmod q}\right|\\
&\ll (\log N)^{2^{C}-1} \sum_{\substack{r\leq N^{1/10}\\r\not \in \mathcal{A}_N}}\frac{d(r)^{C}}{r}\\
&\leq (\log N)^{2^{C}-1} \sum_{r\leq N^{1/10}}\frac{d(r)^{C+1}}{r\cdot 2^{(\log \log N)^{3/2}}}\\
&\ll 2^{-(\log \log N)^{3/2}/2},
\end{align*}
say. In view of this, it suffices to show that for all $r\in \mathcal{A}_N\cap [1,N^{1/10}]$ we have
\begin{align*}
\left|\mathbb{E}_{(N-N(\log N)^{-A})/r<m\leq N/r}g(rm)1_{rm\equiv a\pmod q}\right|\ll_A (\log N)^{-3A/2-1}.    \end{align*}
By writing the sum over $((N-N(\log N)^{-A})/r, N/r]$ as a difference of two sums, it suffices to show that for $y\in [N^{9/10}/2,N]$ we have  
\begin{align}\label{eq:l}
\left|\mathbb{E}_{m\leq y}g(rm)1_{rm\equiv a\pmod q}\right|\ll_A (\log N)^{-A-1}.  
\end{align}
We can uniquely factorise $m=\ell m'$, where $\ell\mid r^{\infty}$ and $(m',r)=1$. By multiplicativity of $g$, we then reduce to showing that
\begin{align}\label{eq:lsum}
\sum_{\ell\mid r^{\infty}}\frac{1}{\ell}\left|\mathbb{E}_{m'\leq y/\ell}g(m')1_{r\ell m'\equiv a\pmod q}1_{(m',r)=1}\right|\ll_A (\log N)^{-3A/2-1}.  
\end{align}

Let $S_1$ denote the part of the left-hand side of~\eqref{eq:lsum} with $\ell<(\log N)^{4A}$, and let $S_2$ denote the part with $\ell\geq (\log N)^{4A}$. By Shiu's bound, we can crudely estimate
\begin{align}\label{eq:linfty}\begin{aligned}
S_2&\ll \sum_{\substack{\ell\mid r^{\infty}\\\ell\geq (\log N)^{3A}}}\frac{1}{\ell}\cdot (\log N)^{2^C-1}\\
&\ll (\log N)^{-4A/2+2^{C}-1}\sum_{\ell\mid r^{\infty}}\frac{1}{\ell^{1/2}}\\
&= (\log N)^{-2A+2^{C}-1}\prod_{p\mid r}\left(1+\frac{1}{p^{1/2}}+\frac{1}{p^{2\cdot 1/2}}+\cdots\right)\\
&\ll (\log N)^{-2A+2^{C}-1}\exp(O(\omega(r)^{1/2})),
\end{aligned}
\end{align}
where for the last line we used the simple inequality 
$$\sum_{p\mid r}\frac{1}{p^{1/2}}\leq \sum_{p\leq \omega(r)}\frac{1}{p^{1/2}}\ll \omega(r)^{1/2}.$$
Since $\omega(r)\leq (\log \log N)^{3/2}$ for $r\in \mathcal{A}_N$, we see that $S_2\ll_A (\log N)^{-3A/2-1}$. 

The remaining task is to show that $S_1\ll_A (\log N)^{-3A/2-1}$. For this, it suffices to show that for any integer $1\leq \ell\leq (\log N)^{4A}$ we have
\begin{align}\label{eq:gl}
\left|\mathbb{E}_{m'\leq y/\ell}g(m')1_{r\ell m'\equiv a\pmod q}1_{(m',r)=1}\right|\ll_A (\log N)^{-3A/2-2}.  
\end{align}
From Shiu's bound and the Siegel--Walfisz assumption on $g$, for any integer $1\leq u\leq N^{1/10}$ and any $r\in \mathcal{A}_N\cap [1,N^{1/10}]$ we have
\begin{align}\label{eq:rn}
 \left|\mathbb{E}_{m'\leq y/\ell}g(m')1_{r\ell m'\equiv a\pmod q}1_{u\mid m'}\right|\ll_{A,C} \min\left\{\frac{(d(u)d(q))^{C}}{u},(\log N)^{-10A}\right\}. 
\end{align}
Substituting the M\"obius inversion formula
\begin{align*}
1_{(m',r)=1}=\sum_{\substack{u\mid r\\u\leq (\log N)^{6A}}}\mu(u)1_{u\mid m'}+\sum_{\substack{u\mid r\\u> (\log N)^{6A}}}\mu(u)1_{u\mid m'}   
\end{align*}
and~\eqref{eq:rn} into~\eqref{eq:gl}, and estimating $d(q)\ll q^{o(1)}=(\log N)^{o(1)}$, we reduce to showing that 
\begin{align*}
 \sum_{\substack{u\mid r\\u>(\log N)^{6A}}}\frac{d(u)^{C}}{u}(\log N)^{2^C-1}\ll_{A,C} (\log N)^{-3A/2-3}.    
\end{align*}
Estimating crudely using $d(u)^{C}/u\ll (\log N)^{-5A/2}/u^{1/2}$ for $u\geq (\log N)^{6A}$, and recalling that $A$ is large, it suffices to show that
\begin{align}\label{eq:2A+1}
 \sum_{u\mid r}\frac{1}{u^{1/2}}\ll \log N, 
 \end{align}
 say.  Since $r\in \mathcal{A}_N$, the left-hand side is
\begin{align*}
\prod_{p\mid r}\left(1+\frac{1}{p^{1/2}}+\frac{1}{p^{2\cdot 1/2}}+\cdots\right)\ll \exp(O(\omega(r^{1/2}))\ll \exp(O((\log \log N)^{3/4})),    
\end{align*}
which suffices. 
\end{proof}

For the proof of Proposition~\ref{prop:SW} we also need a bilinear estimate for polynomial phases.

\begin{lemma}\label{le:bilinear}
Let $A,C>0$, $s\in \mathbb{N}$, and let $\alpha_a$, $\beta_a$ be complex sequences with $|\alpha_a|, |\beta_a|\leq d(a)^{C}$ for $a\in \mathbb{N}$, and with $\alpha_a,\beta_a$ supported on $a\geq \exp((\log \log N)^2)$. Let $P(y)=\sum_{0\leq j\leq s}c_jy^s$ be a polynomial with real coefficients, and suppose that 
\begin{align}\label{eq:alphabeta}
\left|\sum_{ab\leq N}\alpha_a\beta_be(P(ab))\right|\geq N(\log N)^{-A}.    
\end{align}
Then there exists an integer $1\leq \ell\ll_{A,C,s}(\log N)^{O_{A,C,s}(1)}$ such that
\begin{align*}
\|\ell c_j\|\ll_{A,C,s}\frac{(\log N)^{O_{A,C,s}(1)}}{N^j}    
\end{align*}
for all integers $1\leq j\leq s$. 
\end{lemma}

\begin{proof} We may assume that $C>0$ is large and that $A>0$ is large in terms of $C$. Write $\alpha_a=\alpha_a^{(1)}+\alpha_a^{(2)}$ and $\beta_a=\beta_a^{(1)}+\beta_a^{(2)}$, where
\begin{align*}
\alpha_a^{(1)}=\alpha_a1_{|\alpha_a|\leq (\log N)^{A/10}},\quad   \beta_a^{(1)}=\beta_a1_{|\beta_a|\leq (\log N)^{A/10}}.  
\end{align*}
Then, since $A$ is large in terms of $C$, we have 
\begin{align}\label{eq:l2}\begin{aligned}
\sum_{n\leq N}|\alpha_n^{(2)}|^2&\leq \sum_{n\leq N}d(n)^{2C}1_{d(n)^{C}>(\log N)^{A/10}}\\
&\leq (\log N)^{-3A}\sum_{n\leq N}d(n)^{32C}\\
&\ll (\log N)^{-5A/2},
\end{aligned}
\end{align}
and similarly with $\beta_n^{(2)}$ in place of $\alpha_n^{(2)}$. Now, applying the decompositions $\alpha_a=\alpha_a^{(1)}+\alpha_{a}^{(2)}$, $\beta_b=\beta_b^{(1)}+\beta_b^{(2)}$,~\eqref{eq:l2} and Cauchy--Schwarz, the assumption~\eqref{eq:alphabeta} yields
\begin{align*}
\left|\sum_{ab\leq N}\alpha_a'\beta_b'e(P(ab))\right|\gg N(\log N)^{-A/2},      
\end{align*}
where $\alpha_a'=\alpha_a^{(1)}/(\log N)^{A/10}$, $\alpha_b'=\alpha_b^{(1)}/(\log N)^{A/10}$. Since $\alpha_a', \beta_b'$ are $1$-bounded, the result follows e.g. from~\cite[Proposition 2.2]{matomaki-shao} (which is an exponential sum estimate over short intervals; weaker results would also suffice). \end{proof}

\begin{proof}[Proof of Proposition~\ref{prop:SW}] Recall the definition of $\widetilde{g}$ from~\eqref{eq:gtilde}. We take $g_1=\widetilde{g}$, $g_2=g-\widetilde{g}$. Then we immediately have $|g_1|, |g_2|\leq |g|$. In view of Lemma~\ref{le:Lr}, it suffices to show that $\|g_1\|_{u^k[N]}\ll_{A,k} (\log N)^{-A}$.

Let $C>0$ be such that $|g(n)|\leq d(n)^{C}$ for all $n$.  By splitting into short intervals, it suffices to show that for any $A>0$ and any polynomial $P(y)=\sum_{1\leq j\leq k-1}c_jy^j\in \mathbb{R}[y]$ we have
\begin{align}\label{eq:gp}
\left|\mathbb{E}_{N-N(\log N)^{-A}<n\leq N}\widetilde{g}(n)e(P(n))\right|\ll_A (\log N)^{-A/2}.    
\end{align}
Let 
$$\widetilde{g}_N(n)=g(n)(1-1_{p\mid n\implies p\in (Q_N,R_N)}).$$
Write $\mathcal{J}_N=(Q_{N-N(\log N)^{-A}},Q_N]\cup (R_{N-N(\log N)^{-A}},R_N]$ for brevity. Then for $n\in (N-N(\log N)^{-A},N]$ we have 
$$\widetilde{g}(n)=\widetilde{g}_N(n)+O(d(n)^{C}1_{\exists\, p\in \mathcal{J}_N\colon\,\, p\mid n}).$$
Hence, by Shiu's bound, for any $1\leq Y_1<Y_2\leq N$ with $Y_2\geq Y_1+N^{1/2}$ we can estimate
\begin{align}\label{eq:gtildediff}
\mathbb{E}_{Y_1<n\leq Y_2}|\widetilde{g}(n)-\widetilde{g}_N(n)|\ll \sum_{p\in \mathcal{J}_N}\frac{1}{p}\mathbb{E}_{Y_1/p<m\leq Y_2/p}d(pm)^{C}\ll_A (\log N)^{2^C-1-A}.     
\end{align}
Hence, it suffices to prove~\eqref{eq:gp} with $\widetilde{g}_N$ in place of $\widetilde{g}$.  

For $I$ an interval, let $\omega_I(n)$ denote the number of prime factors from $I$ without multiplicities. Then we immediately have the  Ramar\'e identity
\begin{align*}
\widetilde{g}_N(n)=\sum_{p\in (Q_N,R_N)}\,\,\sum_{n=pm}\frac{g(pm)}{\omega_{(Q_N,R_N)}(pm)}.
\end{align*}
By multiplicativity, $g(pm)=g(p)g(m)$ and $\omega_{(Q_N,R_N)}(pm)=1+\omega_{(Q_N,R_N)}(m)$ unless $p\mid m$, so 
\begin{align*}
\widetilde{g}_N(n)&=\sum_{p\in (Q_N,R_N)}\,\,\sum_{n=pm}\frac{g(p)g(m)}{\omega_{(Q_N,R_N)}(m)+1}+O(1_{\exists\, p\in (Q_N,R_N)\colon\,\, p^2\mid n})\\
&\coloneqq g'_N(n)+O(1_{\exists\, p\in (Q_N,R_N)\colon\,\, p^2\mid n}).
\end{align*}
We trivially have
\begin{align}\label{eq:p2}
 \mathbb{E}_{N-N(\log N)^{-A}<n\leq N}1_{\exists\, p\in (Q_N,R_N)\colon\,\, p^2\mid n}\ll \sum_{p\in (Q_N,R_N)}\frac{1}{p^2}\ll \frac{1}{Q_N}.    
\end{align}

Since $g'_N(n)$ is by definition of the form $\sum_{n=ab}\alpha_a\beta_b$ with $|\alpha_a|\leq 1$, $|\beta_b|\ll d(b)^{C}$, and since $\alpha_a,\beta_a$ are supported on $(Q_N,N/Q_N)$ by Lemma~\ref{le:bilinear} and~\eqref{eq:p2} we conclude that
\begin{align*}
\left|\mathbb{E}_{N-N(\log N)^{-A}<n\leq N}\widetilde{g}_N(n)e(P(n))\right|\leq (\log N)^{-A}     
\end{align*}
unless for some integer $1\leq \ell\ll_A (\log N)^{O_A(1)}$ the coefficients of $P$ satisfy
\begin{align}\label{eq:lcj}
\|\ell c_j\|\ll_{A,C,k}\frac{(\log N)^{O_{A,C,k}(1)}}{N^j}    
\end{align}
for all integers $1\leq j\leq s-1$. 

Suppose that~\eqref{eq:lcj} holds. Let $B$ be large in terms of $A, C, k$. Then $e(P(n))=e(P(N))(1+O((\log N)^{-A}))$ for $n$ belonging to any subinterval of $[1,N]$ of length $N(\log N)^{-A}$. Now, splitting into progressions modulo $\ell$ and into short intervals, we see that
\begin{align*}
&\left|\mathbb{E}_{N-N(\log N)^{-A}<n\leq N}\widetilde{g}_N(n)e(P(n))\right|\\
\ll& \ell\max_{b\pmod{\ell}}\,\,\max_{z\in (N-N(\log N)^{-A},N)}\left|\mathbb{E}_{z-N(\log N)^{-B}<n\leq z}\widetilde{g}_N(n)e(P(n))1_{n\equiv b\pmod{\ell}}\right|\\
\ll&  \ell\max_{b\pmod \ell}\,\,\max_{z\in (N-N(\log N)^{-A},N)}\left|\mathbb{E}_{z-N(\log N)^{-B}<n\leq z}\widetilde{g}_N(n)1_{n\equiv b\pmod{\ell}}\right|+O_A((\log N)^{-A}).    
\end{align*}
Now the claim follows from~\eqref{eq:gtildediff} and the Siegel--Walfisz assumption on $\widetilde{g}$ (which we have thanks to Lemma~\ref{le:SW}). 
\end{proof}

\section{Lemmas for the main proofs}

\subsection{A lacunary subsequence trick}

We use a lacunary subsequence trick in the proofs of our pointwise convergence results. Such a trick roughly states that if $A_{N}\colon X\to \mathbb{C}$ are some measurable functions and we have strong quantitative decay for $\|A_{N_j}\|_{L^1(X)}$ for some lacunary sequence $(N_j)_{j\in \mathbb{N}}$, then provided that $A_N$ does not vary too much on intervals of the form $[N_j,N_{j+1}]$ in $L^1(X)$ norm, the sequence $A_N$ must converge to $0$ in $L^{\infty}(X)$ norm.  Variants of this idea are frequently used to establish convergence of ergodic averages; see for example~\cite[Section 5]{fra-les-wierdl}.

\begin{lemma}\label{le:borel-cantelli}
Let $B\geq 40$. Let $(X,\nu)$ be a probability space, and for $N\geq 1$ let $A_N\colon X\to \mathbb{C}$ be a measurable function. If for any $N\geq 3$ we have
    \begin{align}\label{eq1}
    \|A_N\|_{L^1(X)}\ll (\log N)^{-B}   
    \end{align} 
and for $3\leq N<M\leq N(1+(\log N)^{-(B/2-1)})$ and for $\nu$-almost all $x\in X$ we have
 \begin{align}\label{eq2}
|A_M(x)-A_N(x)|\leq (\log N)^{B/10-B/2+1}, 
\end{align}
then for $\nu$-almost all $x\in X$ we have $\lim_{N\to \infty}(\log N)^{2B/5-2}|A_N(x)|=0$. 
\end{lemma}

\begin{proof}
Let $\eta>0$ be a small enough constant, and set  $N_j=\exp(\eta\cdot  j^{2/B})$ for all $j\in \mathbb{N}$. Note that $N_{j+1}\leq N_j(1+(\log N_j)^{-(B/2-1)})$ for all large enough $j\in \mathbb{N}$,
so by~\eqref{eq2} we have
\begin{align}\label{eq2b}
 \sup_{M\in [N_j,N_{j+1}]}|A_{M}(x)-A_N(x)|\leq (\log N_j)^{B/10-B/2+1}   \end{align}
For $\delta\in (0,1/2)$,  denote 
\begin{align}\label{eq:xj}\begin{split}
E(\delta)&=\{x\in X\colon \limsup_{N\to \infty}\,(\log N)^{2B/5-2}|A_N(x)|\geq \delta\},\\
E_{j}(\delta)&=\{x\in X\colon |A_{N_j}(x)|\geq \frac{\delta}{2}(\log N_j)^{-2B/5+2}\}.
\end{split}
\end{align}
Then by~\eqref{eq2b} we have
\begin{align}\label{eq6}
E(\delta)\subset \bigcap_{i=3}^{\infty} \bigcup_{j\geq i} E_{j}(\delta).    
\end{align}

By Markov's inequality and~\eqref{eq1}, we have the bound
\begin{align*}
\nu(E_j(\delta))&\ll_{\delta} (\log N_j)^{2B/5-B-2}\ll j^{-1.1}.
\end{align*}
Since $\sum_{j\geq 3}j^{-1.1}<\infty$, from~\eqref{eq6} and  the Borel--Cantelli lemma we conclude that $\nu(E(\delta))=0$. By the countable additivity of $\nu$, this then implies that 
$$\nu\left(\bigcap_{\delta>0}E(\delta)\right)=\nu\left(\bigcap_{j\geq 1}E(1/j)\right)=0,$$
as desired. 
\end{proof}

\subsection{A simple \texorpdfstring{$L^1$}{L1} bound}

We also need a simple $L^1$ estimate for ergodic averages that follows from H\"older's inequality.

\begin{lemma}\label{le_L1}
Let $k\in \mathbb{N}$. Let $(X,\nu)$ be a probability space, and let $T_1,\ldots, T_k\colon X\to X$ be invertible measure-preserving maps. Let $g_1,\ldots, g_k\colon \mathbb{Z}\to \mathbb{Z}$ be functions. Also let $1<q_1,\ldots,q_k\leq \infty$ satisfy $\frac{1}{q_1}+\cdots +\frac{1}{q_k}\leq   1$. Then, for any $f_1\in L^{q_1}(X),\ldots, f_k\in L^{q_k}(X)$, $\theta\colon \mathbb{N}\to \mathbb{C}$ and $N\geq 1$, we have
\begin{align*}
 \|\mathbb{E}_{n\leq N}\,\theta(n)\prod_{j=1}^k f_j(T_j^{g_j(n)}\cdot)\|_{L^1(X)}\leq \left(\mathbb{E}_{n\leq N}|\theta(n)|^r\right)^{1/r}\|f_1\|_{L^{q_1}(X)}\cdots \|f_k\|_{L^{q_k}(X)},  
\end{align*}
where $1\leq r\leq \infty$ satisfies $1/r+1/q_1+\cdots+1/q_k=1$.
\end{lemma}

\begin{proof}
By H\"older's inequality and the $T$-invariance of $\nu$, we have
\begin{align*}
 &\|\mathbb{E}_{n\leq N}\theta(n)\prod_{j=1}^k f_j(T_j^{g_j(n)}\cdot)\|_{L^1(X)}\\
 &\quad \leq \left(\mathbb{E}_{n\leq N}|\theta(n)|^r\right)^{1/r}\prod_{j=1}^k \left(\int_{X}\mathbb{E}_{n\in \mathcal{S}}|f_j(T_j^{g_j(n)}x)|^{q_j}\d \nu(x)\right)^{1/q_j}\\
 &\quad =\left(\mathbb{E}_{n\leq N}|\theta(n)|^r\right)^{1/r}\prod_{j=1}^k\|f_j\|_{L^{q_j}(X)},
\end{align*}
as claimed.
\end{proof}

\section{Proofs of the pointwise ergodic theorems}\label{sec:mainproof}

All of our main theorems will be proven in the more general setting of weighted polynomial ergodic averages
\begin{align}\label{eq:AN}
A_N^{P_1,\ldots, P_k}(\theta;f_1,\ldots, f_k)(x)\coloneqq \frac{1}{N}\sum_{n\leq N}\theta(n)f_1(T^{P_1(n)}x)\cdots f_k(T^{P_k(n)}x),    
\end{align}
where $\theta\colon \mathbb{N}\to \mathbb{C}$ is a function satisfying suitable uniformity norm estimates. We note the following result does not require $\theta$ to be multiplicative, and the hypotheses are satisfied also for example for $\theta$ being a suitably normalised version of the von Mangoldt function (by~\cite[Theorem 6]{lengII}).

\begin{theorem}[Pointwise convergence of polynomial ergodic averages with nice weight]\label{thm_polyergodic-general}
Let $d,k,K\in \mathbb{N}$. Let $\theta\:\mathbb{N}\to \mathbb{C}$ satisfy 
\begin{align}\label{eq:thetabound}
|\theta(n)|\leq (\log n)^{K}d(n)^{K}
\end{align}
for all $n\geq 3$. 
Let $P_1,\ldots, P_k$ be polynomials with integer coefficients satisfying $\deg P_1\leq \cdots \leq \deg P_k=d$. Suppose that one of the following holds:
\begin{enumerate}[(i)]
    \item We have $$\|\theta\|_{U^{s+1}[M]}\ll_{B,s} (\log M)^{-B}$$ for any $M\geq 3$, $s\in \mathbb{N}$ and $B>0$.
    \item We have $$\|\theta\|_{u^{d+1}[M]}\ll_B (\log M)^{-B}$$ for any $M\geq 3$ and $B>0$, and the polynomials $P_1,\ldots, P_k$ have pairwise distinct degrees. 
\end{enumerate}

Let $(X,\nu,T)$ be a measure-preserving system, and let $1<q_1,\ldots, q_k\leq \infty$ satisfy $\frac{1}{q_1}+\cdots+\frac{1}{q_k}< 1$. Then, for any $f_1\in L^{q_1}(X),\ldots, f_k\in L^{q_k}(X)$ and $A\geq 0$, we have 
\begin{align*}
\lim_{N\to \infty} \frac{(\log N)^{A}}{N}\sum_{n\leq N}\theta(n)f_1(T^{P_1(n)}x)\cdots f_k(T^{P_k(n)}x)=0
\end{align*}
for almost all $x\in X$.
\end{theorem}

Let us first see how this theorem implies our main theorems.

\begin{proof}[Proof of Theorems~\ref{thm_polyergodic} and~\ref{thm_multpolyergodic} assuming Theorem~\ref{thm_polyergodic-general}] Theorem~\ref{thm_polyergodic} follows by taking $\theta=\mu$ and applying Lemma~\ref{le:leng}. 

For proving Theorem~\ref{thm_multpolyergodic}, we first use Proposition~\ref{prop:SW} to obtain a decomposition $g=g_1+g_2$ with $\|g_1\|_{u^k[N]}\ll_{B,k}(\log N)^{-B}$, $\mathbb{E}_{n\leq N}|g_2(n)|^r\ll_r (\log N)^{-1+o(1)}$ and $|g_1|, |g_2|\leq |g|$. Applying Theorem~\ref{thm_polyergodic-general}, we reduce to showing that
\begin{align*}
 \lim_{N\to \infty} \frac{1}{N}\sum_{n\leq N}g_2(n)f_1(T^{P_1(n)}x)\cdots f_k(T^{P_k(n)}x)=0  
\end{align*}
for almost all $x\in X$.

Since $1/q_1+\cdots +1/q_k<1$, we can find some $0<\varepsilon<\min_{1\leq j\leq k}(q_j-1)$ and $r\geq 1$ such that $1/r+1/(q_1-\varepsilon)+\cdots+1/(q_k-\varepsilon)=1$. Then by H\"older's inequality we have 
\begin{align*}
&\left|\mathbb{E}_{n\leq N}g_2(n)f_1(T^{P_1(n)}x)\cdots f_k(T^{P_k(n)}x)\right|\\
&\quad \leq \left(\mathbb{E}_{n\leq N}|g_2(n)|^r\right)^{1/r}\prod_{j=1}^k \left(\mathbb{E}_{n\leq N}|f_j(T^{P_j(n)}x)|^{q_j-\varepsilon}\right)^{1/(q_j-\varepsilon)}.
\end{align*}
In view of the bound $\mathbb{E}_{n\leq N}|g_2(n)|^r\ll_r (\log N)^{-1+o(1)}$, it suffices to show for small enough $\delta>0$ that for all $1\leq j\leq k$ we have 
\begin{align}\label{eq:BC}
\nu(\{x\in X\colon \limsup_{N\to \infty}\,(\log N)^{-\delta}\mathbb{E}_{n\leq N}|f_j(T^{P_j(n)}x)|^{q_j-\varepsilon}\geq 1\})=0.    
\end{align}

Using Markov's inequality and Bourgain's maximal inequality for polynomial ergodic averages (see~\cite[Theorem 1.8(iii)]{krause-mirek-tao}) together with the fact that $f_j^{q_j-\varepsilon}\in L^{q_j/(q_j-\varepsilon)}(X)$, we see that for any $1\leq j\leq k$ we have
\begin{align*}
&\nu(\{x\in X\colon \sup_{N\in [2^{2^m},2^{2^{m+1}}]}\mathbb{E}_{n\leq N}|f_j(T^{P_j(n)}x)|^{q_j-\varepsilon}\geq 2^{(\delta/2)m} \})\\
&\quad \leq (2^{-(\delta/2)m})^{-q_j/(q_j-\varepsilon)}\int_X \left(\sup_{N\geq 1}\mathbb{E}_{n\leq N}|f_j(T^{P_j(n)}x)|^{q_j-\varepsilon}\right)^{q_j/(q_j-\varepsilon)} \d \nu(x)\\
&\quad \ll 2^{-q_j\delta m/(2(q_j-\varepsilon))}\|f_j\|_{L^{q_j}(X)}^{q_j}.  \end{align*}
Since $\sum_{m\geq 1}2^{-cm}<\infty$ for any $c>0$, the claim~\eqref{eq:BC} follows from the above estimate and the Borel--Cantelli lemma. 
\end{proof}

We will reduce Theorem~\ref{thm_polyergodic-general} to the following quantitative $L^1$ estimate.

\begin{proposition}\label{prop:L1polyergodic}
Let the notation and assumptions be as in Theorem~\ref{thm_polyergodic-general}. Then, for all $N\geq 3$, we have
\begin{align}\label{eq:nubound}
\|A_N^{P_1,\ldots,P_k}(\theta;f_1,\ldots, f_k)\|_{L^1(X)}\ll_{A,K,P_1,\ldots, P_k} \left(\prod_{j=1}^k(1+\|f_j\|_{L^{q_j}(X)})\right)^{B(q_1,\ldots, q_k)}(\log N)^{-A}    
\end{align}
for some $B(q_1,\ldots, q_k)>0$.
\end{proposition}

\begin{proof}[Proof that Proposition~\ref{prop:L1polyergodic} implies Theorem~\ref{thm_polyergodic-general}] We are going to apply Lemma~\ref{le:borel-cantelli}. Recall the notation~\eqref{eq:AN}. Let $r\geq 1$ satisfy $1/r+1/q_1+\cdots +1/q_k=1$. Applying first the triangle inequality, then Lemma~\ref{le_L1} and finally~\eqref{eq:thetabound} and Shiu's bound, we see that for any $B>0$, $N\geq 3$ and $M\in [N,(1+(\log N)^{-(B/2-1)})N]$ we have
\begin{align*}
&\|A_M^{P_1,\ldots, P_k}(\theta;f_1,\ldots, f_k)-A_N^{P_1,\ldots, P_k}(\theta;f_1,\ldots, f_k)\|_{L^{1}(X)}\\
&\quad \leq \int_X  \left(\frac{1}{N}\sum_{N\leq n\leq M}|\theta(n)|\prod_{j=1}^k|f_j(T^{P_j(n)} x)|+\left(\frac{1}{N}-\frac{1}{M}\right)\sum_{n\leq N}|\theta(n)|\prod_{j=1}^k|f_j(T^{P_j(n)} x)|\right)   \d \nu(x)\\
&\quad \ll \left(\frac{M-N}{N}\left(\mathbb{E}_{N\leq n<M}|\theta(n)|^r\right)^{1/r}+\frac{(\log N)^{-(B/2-1)}}{N}\left(\mathbb{E}_{n\leq N}|\theta(n)|^r\right)^{1/r}\right)\prod_{j=1}^k\|f_j\|_{L^{q_j}(X)}\\
&\ll (\log N)^{-B/2+1+K+(2^{Kr}-1)/r}\prod_{j=1}^k\|f_j\|_{L^{q_j}(X)}.
\end{align*}
 If $B$ is large enough, the exponent of the logarithm above is at most $-0.49B$, say. Hence, by Lemma~\ref{le:borel-cantelli}, the conclusion of Theorem~\ref{thm_polyergodic-general} follows from~\eqref{eq:nubound}.
\end{proof}

\begin{proof}[Proof of Proposition~\ref{prop:L1polyergodic}] 
We first reduce the proof of Proposition~\ref{prop:L1polyergodic} to the case $q_1=\cdots=q_k=\infty$. 

\textbf{Reduction to the case of bounded functions.} We claim that it suffices to prove Proposition~\ref{prop:L1polyergodic} in the case $q_1=\cdots=q_k=\infty$. Suppose that this case has been proven. 
Let $1<q_1,\ldots, q_k<\infty$ and $f_i\in L^{q_i}(X)$ for $i\in \{1,\ldots, k\}$ (we can assume that all the $q_i$ are $<\infty$, since we have $\|f\|_{L^q(X)}\leq \|f\|_{L^{\infty}(X)}$ for any $q<\infty$ and any measurable $f$). Also let $A>0$, and let $C>0$ be large enough in terms of $A,K,q_1,\ldots q_k$.

For $j\in\{1,\ldots, k\}$ and $N\geq 2$, we split
\begin{align*}
f_j=f_{j,N}^{(1)}+f_{j,N}^{(2)},\quad \textnormal{where}\quad f_{j,N}^{(1)}(x)=f_{j}(x)1_{|f_j(x)|\leq (\log N)^{C}},\,\, f_{j,N}^{(2)}(x)=f_{j}(x)1_{|f_j(x)|>(\log N)^{C}}.     
\end{align*}
Then by linearity we see that
\begin{align}\label{eq:AN-linear}\begin{split}
&A_N^{P_1,\ldots, P_k}(\theta; f_1, \ldots, f_k)\\
\quad &=A_N^{P_1,\ldots, P_k}(\theta;f_{1,N}^{(1)},\ldots, f_{k,N}^{(1)})+\sum_{(i_1,\ldots, i_k)\in [2]^k\setminus (1,\ldots, 1)}A_N^{P_1,\ldots, P_k}(\theta;f_{1,N}^{(i_1)},\ldots, f_{k,N}^{(i_k)}).    
\end{split}
\end{align}
Since $\|(\log N)^{-C}f_{j,N}^{(1)}\|_{L^{\infty}(X)}\leq 1$ for $j\in \{1,\ldots, k\}$, by the case $q_1=\cdots=q_k=\infty$ of Proposition~\ref{prop:L1polyergodic} we have
\begin{align*}
\|A_N^{P_1,\ldots, P_k}(\theta;f_{1,N}^{(1)},\ldots, f_{k,N}^{(1)})\|_{L^1(X)}\ll_{A,C,K,P_1,\ldots, P_k} (\log N)^{-A}. \end{align*}

To bound the error term in~\eqref{eq:AN-linear}, it suffices to show that for $(i_1,\ldots, i_k)\in [2]^k\setminus (1,\ldots, 1)$ we have 
\begin{align*}
&\|A_N^{P_1,\ldots, P_k}(\theta;f_{1,N}^{(i_1)},\ldots, f_{k,N}^{(i_k)})\|_{L^{1}(X)}\ll_{C,q_1,\ldots, q_k} (\log N)^{K+(2^{Kr}-1)/r-\delta C}\prod_{j=1}^k (1+\|f_j\|_{L^{q_j}(X)}^{1+\delta})  \end{align*}
for some $\delta=\delta(q_1,\ldots, q_k)>0$, since $C$ can be taken to be large enough in terms of $A,K,q_1,\ldots, q_k$ so that $K(2^{Kr}-1)/r-\delta C\leq -A$.

For $(i_1,\ldots, i_k)\in [2]^k\setminus (1,\ldots, 1)$, there is some $\ell$ such that $i_{\ell}=2$; for the sake of notation, assume that $\ell=1$. 
Fix some $\delta>0$ and $1\leq r<\infty$ such that  $\frac{1}{r}+\frac{1+\delta}{q_1}+\frac{1}{q_2}+\cdots+\frac{1}{q_k}=1$. 
Using the triangle inequality, Lemma~\ref{le_L1} and Shiu's bound combined with the assumption $|\theta(n)|\leq (\log n)^{K}d(n)^K$, we obtain
\begin{align*}
&\|A_N^{P_1,\ldots, P_k}(\theta;f_{1,N}^{(i_1)},\ldots, f_{k,N}^{(i_k)})\|_{L^{1}(X)}\\
&\quad \leq (\log N)^{-\delta C}\|A_N^{P_1,\ldots, P_k}(\theta;|f_{1,N}^{(i_1)}|^{1+\delta},\ldots, |f_{k,N}^{(i_k)}|)\|_{L^{1}(X)}\\
&\quad \leq (\log N)^{-\delta C}\left(\mathbb{E}_{n\leq N}|\theta(n)|^r\right)^{1/r}\|f_1\|_{L^{q_1}(X)}^{1+\delta}\cdots\|f_k\|_{L^{q_k}(X)}\\
&\quad \ll_{K,r}(\log N)^{K+(2^{Kr}-1)/r-\delta C}\|f_1\|_{L^{q_1}(X)}^{1+\delta}\cdots\|f_k\|_{L^{q_k}(X)}.
\end{align*}
 Now the proof of Proposition~\ref{prop:L1polyergodic} has been reduced to the case $q_1=\cdots=q_k=\infty$.

\textbf{The case of bounded functions.} It now remains to show~\eqref{eq:nubound} in the case $q_1=\cdots=q_k=\infty$. In the rest of the proof, we abbreviate $A_N(x)\coloneqq A_N^{P_1,\ldots, P_k}(\theta;f_1,\ldots, f_k)(x)$.

By Cauchy--Schwarz, it suffices to show that
\begin{align*}
\int_{X}|A_N(x)|^2\d \nu(x)\ll_{A,K,P} \|f_1\|_{L^{\infty}(X)}^2\cdots\|f_k\|_{L^{\infty}(X)}^2(\log N)^{-A}.     
\end{align*}
The assumption $|\theta(n)|\leq (\log n)^{K}d(n)^{K}$ gives $$\|A_N\|_{L^{\infty}(X)}\ll (\log N)^{K+2^K-1}\|f_1\|_{L^{\infty}(X)}\cdots \|f_k\|_{L^{\infty}(X)},$$
so it suffices to show that for any  $\phi\in L^{\infty}(X)$ we have
\begin{align}\label{eq:phiAn}
\int_{X}\phi(x)A_N(x)\d \nu(x)\ll_{A,K,P}  \|f_1\|_{L^{\infty}(X)}\cdots \|f_k\|_{L^{\infty}(X)} \|\phi\|_{L^{\infty}(X)}(\log N)^{-A}.        
\end{align} 
By the $T$-invariance of $\nu$, the claim~\eqref{eq:phiAn} is equivalent to
\begin{align}\label{eq:desired}\begin{split}
&\int_{X}\frac{1}{N^{d+1}}\sum_{m\in [N^d],n\in [N]}\theta(n)\phi(T^m x)f_1(T^{m+P_1(n)}x)\cdots f_k(T^{m+P_k(n)}x)\d \nu(x)\\
&\quad \ll_{A,K,P_1,\ldots, P_k}   \|f_1\|_{L^{\infty}(X)}\cdots \|f_k\|_{L^{\infty}(X)}\|\phi\|_{L^{\infty}(X)}(\log N)^{-A}. 
\end{split}
\end{align}

By the definition of the $L^{\infty}(X)$ norm, there exists a set $X'\subset X$ such that
\begin{align*}
\nu(X')=1\quad \textnormal{and}\quad |\phi(x)|\leq \|\phi\|_{L^{\infty}(X)},\,|f_i(x)|\leq \|f_i\|_{L^{\infty}(X)}\,\, \textnormal{for all}\,\, 1\leq i\leq k,\, x\in X'. \end{align*}
Restricting the integral in~\eqref{eq:desired} to $X'$, it suffices to show that for all $x\in  X'$ we have the bound
\begin{align}\label{eq:pointwise1}\begin{split}
&\left|\frac{1}{N^{d+1}}\sum_{m\in [N^d],n\in [N]}\frac{\theta(n)}{(\log N)^{A}}\phi(T^m x)f_1(T^{m+P_1(n)}x)\cdots f_k(T^{m+P_k(n)}x)\right|\\
&\quad \ll_{A,P_1,\ldots, P_k}  \|f_1\|_{L^{\infty}(X)}\cdots \|f_k\|_{L^{\infty}(X)}\|\phi\|_{L^{\infty}(X)}(\log N)^{-2A}. 
\end{split}
\end{align}
Write $\theta=\theta_1+\theta_2$ where $\theta_1(n)=\theta(n)1_{|\theta(n)|\leq (\log n)^{A}}$, $\theta_2(n)=\theta(n)1_{|\theta(n)|> (\log n)^{A}}$. Since $d(n)^{K}\geq (\log n)^{A-K}$ in the support of $\theta_2$, we have
\begin{align}\label{eq:alogloga}\begin{aligned}
\mathbb{E}_{n\leq N}|\theta_2(n)|&\ll (\log N)^{K}\mathbb{E}_{n\leq N}(\log(n+1))^{-(A-K)\log \log A}d(n)^{K+K\log \log A}\\
&\ll (\log N)^{-(A\log \log A)/2}    
\end{aligned}
\end{align}
if $A$ is large enough in terms of $K$.

Now, if assumption (i) of the theorem holds, by the triangle inequality and~\eqref{eq:alogloga}  we have
\begin{align}\label{eq:Us1}\begin{aligned}
 \|\theta_1\|_{U^s[N]}\leq \|\theta\|_{U^s[N]}+\|\theta_2\|_{U^s[N]}\leq \|\theta\|_{U^s[N]}+\|\theta_2\|_{U^1[N]}\ll_{A,s,P_1,\ldots, P_k} (\log N)^{-(A\log \log A)/2}. 
 \end{aligned}
\end{align}
If instead assumption (ii) holds, we similarly have
\begin{align}\label{eq:us1}
 \|\theta_1\|_{u^s[N]}\ll_{A,s,P_1,\ldots, P_k} (\log N)^{-(A\log \log A)/2}.  
\end{align}
In either case, since we may assume that  $\phi,f_1,\ldots, f_k$ are supported on $[-CN^,CN^d]$ for some $C\ll_{P} 1$ and since $|\theta_1(n)|(\log n)^{-A}\leq 1$, we can use either Theorem~\ref{thm:GVNT1} or~\ref{thm:GVNT2} (depending on whether we have~\eqref{eq:us1} or~\eqref{eq:Us1}) to conclude that~\eqref{eq:pointwise1} holds with $\theta_1$ in place of $\theta$. Hence it suffices to show that~\eqref{eq:pointwise1} holds with $\theta_2$ in place of $\theta$. For this it suffices to show that
\begin{align*}
\mathbb{E}_{n\leq N}|\theta_2(n)|\ll (\log N)^{-A}.     
\end{align*}
Here the left-hand side is
\begin{align*}
\leq (\log N)^{K-2(A-K)}\mathbb{E}_{n\leq N}d(n)^{2K}\ll (\log N)^{-A} 
\end{align*}
if $A$ is large enough in terms of $K$. This completes the proof. 
\end{proof}

Lastly, we prove Theorem~\ref{thm_commuting}.

\begin{proof}[Proof of Theorem~\ref{thm_commuting}] Following the proof of Theorem~\ref{thm_polyergodic-general} verbatim, we reduce matters to showing that
\begin{align}\label{eq:pointwise1b}\begin{split}
&\left|\frac{1}{N^{k+1}}\sum_{m_1,\ldots,m_k,n\in [N]}\mu(n)\phi(T_1^{m_1}\cdots T_k^{m_k} x)\prod_{j=1}^k f_j(T_1^{m_1}\cdots T_k^{m_k}T_j^nx)\right|\\
&\quad \ll_{A}  \|f_1\|_{L^{\infty}(X)}\cdots \|f_k\|_{L^{\infty}(X)}\|\phi\|_{L^{\infty}(X)}(\log N)^{-A}
\end{split}
\end{align}
for any functions $\phi,f_1,\ldots, f_k\in L^{\infty}(X)$ and any $x\in X$ for which $|\phi(x)|\leq \|\phi\|_{L^{\infty}(X)}, |f_j|\leq \|f_j\|_{L^{\infty}(X)}$. By setting \begin{align*}
F_j((m_1,\ldots, m_k))&=f_j(T_1^{m_1}\cdots T_k^{m_k}x),\\
G((m_1,\ldots, m_k))&=\phi(T_1^{m_1}\cdots T_k^{m_k}x),
\end{align*}
the estimate~\eqref{eq:pointwise1b} reduces to
\begin{align*}
&\left|\frac{1}{N^{k+1}}\sum_{\mathbf{m}\in [N]^k,n\in [N]}\mu(n)G(\mathbf{m})F_1(\mathbf{m}+(n,0,\ldots, 0))\cdots F_k(\mathbf{m}+(0,\ldots, 0,n))\right|\\
&\quad \ll_{A}  \|F_1\|_{L^{\infty}(X)}\cdots \|F_k\|_{L^{\infty}(X)}\|G\|_{L^{\infty}(X)}(\log N)^{-A}.
\end{align*}
But this bound follows directly from Lemmas~\ref{le:GVNT-multi} and~\ref{le:leng}.
\end{proof}

\bibliography{refs}
\bibliographystyle{plain}

\end{document}